\documentclass[review,3p,sort]{elsarticle}

\usepackage[hidelinks]{hyperref}
\usepackage{amsmath,amssymb,amsthm}
\usepackage{mathtools}	
\usepackage{stmaryrd} 
\usepackage{accents}
\usepackage{subfig}
\usepackage{booktabs}
\usepackage{epstopdf}
\usepackage[utf8]{inputenc}
\usepackage{longtable}
\usepackage{arydshln} 
\usepackage{xargs}    

\usepackage[colorinlistoftodos,prependcaption,textsize=tiny]{todonotes}
\newcommandx{\unsure}[2][1=]{\todo[linecolor=blue,backgroundcolor=blue!25,bordercolor=blue,#1]{#2}}
\newcommandx{\changeThis}[2][1=]{\todo[linecolor=red,backgroundcolor=red!25,bordercolor=red,#1]{#2}}

\newtheorem{theorem}{Theorem}[section]
\newtheorem{lemma}[theorem]{Lemma}

\theoremstyle{definition}

\theoremstyle{remark}
\newtheorem{remark}[theorem]{Remark}
\newcommand\oneHalf{\frac{1}{2}}
\newcommand{\geom}{\textrm{geo}}
\newcommand{\geo}[1]{\ensuremath{\left\{\hspace*{-3pt}\left\{#1\right\}\hspace*{-3pt}\right\}^\geom}}

\newcommand\spacevec[1]{\accentset{\,\rightarrow}{#1}}		
\newcommand\statevec[1]{\mathbf #1}					

\newcommand\mmatrix[1]{\underbar{\textnormal{#1}}} 

\newcommand\dS{\,\operatorname{dS} }

\newcommand\ent{\varepsilon}
\newcommand\ext{\textrm{ext}}
\newcommand\energy{\mathcal{E}}

\newcommand\TMat{\mmatrix{T}} 
\newcommand\SMat{\mmatrix{S}} 
\newcommand\MMat{\mmatrix{M}} 
\newcommand\invMMatT{\mmatrix{M}^{-T}} 
\newcommand\NMat{\mmatrix{N}} 
\newcommand\IMat{\mmatrix{I}} 
\newcommand\AMat{\mmatrix{A}} 
\newcommand\IMinus{\IMat^{-}} 
\newcommand\IPlus{\IMat^{+}} 

\newcommand\LambMat{\mmatrix{$\Lambda$}} 
\newcommand\LambPlus{\LambMat^{+}} 
\newcommand\LambMinus{\LambMat^{-}} 
\newcommand{\absLambM}{|\LambMinus|}
\newcommand{\sqrtLambM}{\sqrt{\absLambM}}
\newcommand{\absLambMExt}{|\LambMinus_{\ext}|}

\newcommand\Wvec{\statevec{W}} 
\newcommand\Wplus{\Wvec^{+}} 
\newcommand\Wminus{\Wvec^{-}} 
\newcommand\WminusExt{\Wvec^{-}_{\ext}} 

\newcommand\Zvec{\statevec{Z}} 

\newcommand\Uvec{\statevec{U}}  
\newcommand\Uprimvec{\statevec{U}_{prim}}  
\newcommand\Vvec{\statevec{V}}  
\newcommand\Uext{\statevec{U}_{\ext}} 

\newcommand\Gvec{\statevec{G}} 

\newcommand\hext{h_{\ext}}
\newcommand\uext{u_{\ext}}
\newcommand\cext{c_{\ext}}

\newcommand\vnext{v_n^{\ext}}
\newcommand\vtext{v_{\tau}^{\ext}}

\newcommand\revOne[1]{\textcolor{black}{#1}}
\newcommand\revTwo[1]{\textcolor{black}{#1}}

\journal{Journal of Computational Physics}

\newcommand{\orcid}[1]{\href{https://orcid.org/#1}{\texorpdfstring{\includegraphics[width=10pt]{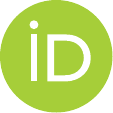}}}}

\numberwithin{equation}{section}

\begin{document}

\begin{frontmatter}


\title{Numerical boundary flux functions that give provable bounds for nonlinear initial boundary value problems with open boundaries}

\author[liu]{Andrew R. Winters\corref{cor1}\orcid{0000-0002-5902-1522}}
\ead{andrew.ross.winters@liu.se}
\cortext[cor1]{Corresponding author}

\author[fl,ca]{David A. Kopriva\orcid{0000-0002-8076-0856}}

\author[liu,sa]{Jan Nordstr\"{o}m\orcid{0000-0002-7972-6183}}

\address[liu]{Department of Mathematics, Applied Mathematics, Link\"{o}ping University, 581 83 Link\"{o}ping, Sweden}

 \address[fl]{Department of Mathematics, The Florida State University, Tallahassee, FL 32306, USA}

\address[ca]{Computational Science Research Center, San Diego State University, San Diego, CA, USA}

\address[sa]{Department of Mathematics and Applied Mathematics, University of Johannesburg,
 P.O. Box 524, Auckland Park 2006, South Africa.}
 
\begin{abstract}
We present a strategy for interpreting nonlinear, characteristic-type penalty terms as numerical boundary flux functions that provide provable bounds for solutions to nonlinear hyperbolic initial boundary value problems with open boundaries.
This approach is enabled by recent work that found how to express the entropy flux as a quadratic form defined by a symmetric boundary matrix.
The matrix formulation provides additional information for how to systematically design characteristic-based penalty terms for the weak enforcement of boundary conditions.
A special decomposition of the boundary matrix is required to define an appropriate set of characteristic-type variables.
The new boundary fluxes are directly compatible with high-order accurate split form discontinuous Galerkin spectral element and similar methods and guarantee that 
the solution is entropy stable and bounded solely by external data.
We derive inflow-outflow boundary fluxes specifically for the Burgers equation and the two-dimensional shallow water equations, which are also energy stable.
Numerical experiments demonstrate that the new nonlinear fluxes do not fail in situations where standard boundary treatments based on linear analysis do.
\end{abstract}


\end{frontmatter}







\section{Introduction}



    Discontinuous Galerkin Spectral Element Methods (DGSEMs) are implemented in numerous packages \cite{KRAIS2021186,schlottkelakemper2021purely,FERRER2023108700,rueda2021entropy,WARUSZEWSKI2022111507}. The methods are attractive because they feature geometric flexibility with high order accuracy. They can also be stabilized by bounding some measure of the solution (e.g. entropy \cite{tadmor1984,Tadmor2003} and/or kinetic energy \cite{jameson2008,pirozzoli2011}) 
    provided that a split form approximation is used in the interior, and that the boundary conditions are implemented stably, e.g. \cite{gassner2018br1,gassner2016split,winters2021}. True energy stability, which requires a bound on all dependent variables, has not been previously shown except when the mathematical entropy is also an energy \cite{nordstrom2022nonlinear,nordstrom2022linear}\revOne{, otherwise additional positivity assumptions as in \cite{tadmor1984,Tadmor2003,svard2025entropy} are necessary}.


    As with Finite Volume schemes, boundary conditions for DGSEMs are implemented through a numerical boundary flux in the form of a Riemann solver by specifying an external state as the data. Wall boundary conditions for the compressible Euler equations of gas-dynamics, for instance, specify the external state as a reflection of the internal one, with no external data required. They have been shown to be entropy bounded depending on the Riemann solver chosen \cite{hindenlangWall}. Inflow-outflow boundary conditions, which are needed in limited area computations and require external data, are usually set by specifying the state variables whose values one wants to enforce.

    The stability of nonlinear inflow-outflow boundary conditions has rarely been studied theoretically. For linear hyperbolic systems, characteristic-type boundary conditions have been shown to be energy stable for general summation-by-parts (SBP) based schemes such as finite difference \cite{sbp1,nordstrom2017roadmap}, finite volume \cite{nordstrom2012weak,nordstrom2003finite}, spectral elements \cite{carpenter2014entropy,carpenter1996spectral}, flux reconstruction \cite{ranocha2016summation}, discontinuous Galerkin \cite{Gassner:2013ij,kopriva2021} and continuous Galerkin schemes \cite{abgrall2020analysis,Hanifa2025},
    meaning that the energy rate is bounded solely in terms of the specified boundary data. However, for non-linear problems, it is not known whether any given numerical boundary flux can ensure that the solution is bounded only by external data, and it has not been clear how to define a numerical boundary flux that does.

    Properly incorporating data into solution bounds is the 
    most crucial part of the analysis at open boundaries.
    In a series of articles, Nordstr\"{o}m~\cite{nordstrom2022nonlinear,nordstrom2023nonlinear,nordstrom2025open,nordstrom2025linear} examined skew-symmetric formulations and applied a non-linear energy method to bound the solution with data under minimal assumptions.
    This nonlinear energy analysis identified nonlinear characteristic-type variables.
    Nonlinearly stable boundary conditions were weakly imposed by penalizing the incoming (nonlinear) characteristic with appropriately scaled data.

    Here, we derive and prove a condition on numerical boundary fluxes for inflow-outflow boundaries that ensures that the mathematical entropy rate is bounded solely by boundary data \revOne{without additional assumptions on solution quantities}. We then derive explicit examples of such flux functions that satisfy these conditions for the Burgers equation and for the shallow water equations, for which entropy boundedness is also energy stability. Finally, we present numerical examples using a DGSEM that show that standard flux functions may work, but can fail, whereas the new fluxes do not.

\section{Brief overview of entropy analysis}\label{sec:entropy}

 To outline the entropy analysis, we consider the system of nonlinear hyperbolic conservation laws 
 \begin{equation}
     \statevec{q}_t + \statevec{f}_x = 0,
     \label{eq:ConservationLaw}
 \end{equation}
where $\statevec{q}$ is the state vector of conserved variables and $\statevec{f}$ is the flux.
The system is solved in a domain $\Omega$ with initial conditions $\statevec{q}(x,0) = \statevec{q}_0(x)$
and boundary conditions $\statevec{q} = \statevec{q}_{\text{ext}}(x,t)$ in $\partial\Omega$.
 
The conservative system \eqref{eq:ConservationLaw} possesses auxiliary conserved quantities.
Notably, the mathematical entropy $S(\statevec{q})$, which is a strictly convex function of the conservative variables.
To reveal the additional conservation law we define the set of entropy variables $\statevec{v} = S_{\statevec{q}}$,
contract \eqref{eq:ConservationLaw} from the left with the entropy variables, and integrate over the domain to have
\begin{equation}
   \int\limits_\Omega\statevec{v}^T\left(\statevec{q}_t + \statevec{f}_x\right)\,\mathrm{d}\Omega = 0.
   \label{eq:weakform}
\end{equation}
From the chain rule, the first term above becomes the time derivative of the entropy function, i.e. $\statevec{v}^T\statevec{q}_t=S_t$. 
On the spatial derivative term in (\ref{eq:weakform}) we integrate-by-parts to generate boundary terms
\begin{equation}
   \int\limits_\Omega S_t \,\mathrm{d}\Omega + \int\limits_{\partial\Omega}\statevec{v}^T\statevec{f} \dS - \int\limits_\Omega \statevec{v}_x^T\statevec{f}\,\mathrm{d}\Omega = 0.
\end{equation}
To weakly incorporate the boundary conditions into the entropy evolution, we introduce a modified flux that is a function of the internal solution state and the external data, $\statevec{f}^*(\statevec{q}, \statevec{q}_{\text{ext}})$, to have
\begin{equation}
   \int\limits_\Omega S_t \,\mathrm{d}\Omega + \int\limits_{\partial\Omega}\statevec{v}^T\statevec{f}^*(\statevec{q}, \statevec{q}_{\text{ext}})\dS - \int\limits_\Omega \statevec{v}_x^T\statevec{f}\,\mathrm{d}\Omega = 0.
\end{equation}
We then integrate-by-parts again to obtain a penalty-type weak imposition of the boundary conditions
\begin{equation}
   \label{eq:entIntermediate}
   \int\limits_\Omega S_t \,\mathrm{d}\Omega + \int\limits_\Omega \statevec{v}^T\statevec{f}_x\,\mathrm{d}\Omega +  \int\limits_{\partial\Omega} \statevec{v}^T\left(\statevec{f}^*(\statevec{q}, \statevec{q}_{\text{ext}}) - \statevec{f}\right)\dS = 0.
\end{equation}
For the middle term above, we invoke that the entropy function $S$ has an associated entropy flux $f^\ent$ that satisfies the compatibility condition \cite{tadmor1984,Tadmor2003} 
$\statevec{v}^T\statevec{f}_{\statevec{q}} = f^\ent_{\statevec{q}}$
so that 
$\statevec{v}^T\statevec{f}_x = \statevec{v}^T\statevec{f}_{\statevec{q}} \statevec{q}_x= f^\ent_{\statevec{q}}\statevec{q}_x = f^\ent_x$.
Thus, the volume contributions of the flux in \eqref{eq:entIntermediate} move onto the boundary to obtain
\begin{equation}
   \label{eq:entIntermediate2}
   \int\limits_\Omega S_t \,\mathrm{d}\Omega +  \int\limits_{\partial\Omega}\left\{f^\ent + \statevec{v}^T\left(\statevec{f}^*(\statevec{q}, \statevec{q}_{\text{ext}}) - \statevec{f}\right)\right\}\dS= 0.
\end{equation}
Finally, we define the rate of change of the total mathematical entropy over the domain by $\energy_t \equiv \int_\Omega S_t \,\mathrm{d}\Omega$
and obtain the final form of the entropy evolution
\begin{equation}
   \label{eq:entFinal}
   \frac{d\energy}{dt} +  \int\limits_{\partial\Omega}\left\{f^\ent + \statevec{v}^T\left(\statevec{f}^*(\statevec{q}, \statevec{q}_{\text{ext}}) - \statevec{f}\right)\right\}\dS= 0.
\end{equation}

From \eqref{eq:entFinal}, we observe that the mathematical entropy can be bounded by its exchange across the physical boundaries.
For instance, in a periodic domain, all boundary terms cancel, yielding $\energy_t = 0$.
That is, the mathematical entropy is conserved.
Although the approach above is somewhat unconventional, introducing a weak imposition of the boundary condition through a modified flux function, $\statevec{f}^*$, highlights the crucial role such fluxes play in bounding the mathematical entropy.
The design of this modified flux function that will eventually become the numerical flux function will be the focus of the remainder of this work.

\section{The split form discontinuous Galerkin spectral element method}\label{sec:DGSEM}

The discontinuous Galerkin spectral element method (DGSEM) is an approximation strategy for hyperbolic conservation laws.
We provide a broad summary of the basic details to construct a nodal DGSEM where complete details can be found in \cite{kopriva2009implementing,winters2021}.
The nodal DGSEM is built from the weak form of the equations where Lagrange interpolating polynomials of degree $N$ are the test functions.
The solution and fluxes are also approximated by polynomials of degree $N$.
Coupling between elements as well as weakly enforcing boundary conditions in the DGSEM is done via a numerical flux function in the normal direction $\statevec{F}_n^*$.
This grants the DGSEM flexibility, as different numerical flux functions can be chosen at internal interfaces and physical boundaries.
The numerical flux at physical boundaries, ideally, enforces dissipative boundary conditions that guarantee stability bounds in terms of external data.
However, numerical flux functions that weakly impose dissipative open boundary conditions for nonlinear problems are lacking.
The design of such boundary flux functions for nonlinear problems that guarantee the solution energy is bounded by external data is the focus of this paper.

Integrals in the weak form are approximated by Legendre-Gauss-Lobatto (LGL) quadrature where the quadrature nodes are collocated with the interpolation nodes for computational efficiency.
This choice of collocation and LGL quadrature results in an approximation with discrete integration and derivative matrices that satisfy the summation-by-parts (SBP) property \cite{gassner2013skew} for any polynomial order.
The movement of flux contributions out of the volume and onto the surface, necessary for the manipulations in Sec.~\ref{sec:entropy}, is possible via derivative approximations that satisfy the SBP property.
As such, the nodal DGSEM on LGL nodes can discretely recover properties, like entropy conservation or kinetic energy preservation, with appropriate choices of split form fluxes.

When split form fluxes are used, the boundedness of the DGSEM approximation depends only on the boundary conditions. Mapped to the reference element, $E$, the total mathematical entropy, $\mathcal{E}$, on an element satisfies an equation that depends only on boundary terms,
\begin{equation}
   \frac{d\energy}{dt} + \int\limits_{\partial E,N} \left\{F_n^{\ent} + \Vvec^T\left(\statevec{F}^*_n - \statevec{F}_n\right)\right\}\hat{s}\dS \le 0.
   \label{eq:EntropyBoundaryIntegral}
\end{equation}
In \eqref{eq:EntropyBoundaryIntegral}, $F_n^{\ent}$ is the entropy flux and $\statevec V$ is the entropy variable state vector.
Polynomials of degree $N$ approximate the normal fluxes $\statevec{F}^*_n$ for the numerical boundary flux (Riemann solver) and $\statevec{F}_n$ for the local normal flux.
The subscript $N$ on the integral symbol denotes the approximation of the integral by Gauss-Lobatto quadrature.
The quantity $\hat s$ is the scaling between the normals in physical and reference space.
Between elements, the entropy conservative or bounded numerical fluxes lead to the bound \eqref{eq:EntropyBoundaryIntegral} that depends only on the fluxes at physical boundaries ~\cite{gassner2016split,winters2021}.

The mathematical entropy is a convex function of the dependent variables associated with a conservation law for smooth solutions.
As previously discussed, this auxiliary evolution equation is not explicitly built into the original equations.
Instead, the form of the mathematical entropy conservation law and its fluxes are determined by contracting with an appropriate set of variables, in this case, $\statevec V$, together with compatibility conditions~\cite{tadmor1984}.
In the examples presented in Secs.~\ref{sec:burgers} and \ref{sec:swe} the total energy is a mathematical entropy function, so, in the following, we will use entropy to denote both quantities.

There is a one-to-one discrete analog comparing \eqref{eq:EntropyBoundaryIntegral} with the continuous result \eqref{eq:entFinal}.
The entropy is bounded if the integrand along physical boundaries in \eqref{eq:EntropyBoundaryIntegral} is bounded by data, and the only freedom is in the choice of the numerical boundary flux.
\revOne{As mentioned in the introduction, for open boundaries, provably bounded nonlinear numerical flux functions for the shallow water equations without additional stabilizing assumptions are unavailable. }
It is not obvious how to discern directly how the entropy flux, $F_n^{\ent}$, can correctly ``balance'' with the numerical flux penalty, $\statevec{F}^*_n - \statevec{F}_n$, to ensure that the boundary integral is bounded by data. To find such a balance is the topic of this paper.

\section{A condition for nonlinearly bounded boundary flux functions}\label{sec:generic}



We derive a condition on the flux function $\statevec F^*$ in \eqref{eq:EntropyBoundaryIntegral} to guarantee entropy boundedness by collecting results from a series of papers on nonlinearly stable boundary conditions found in \cite{nordstrom2022nonlinear,nordstrom2023nonlinear,nordstrom2025open,nordstrom2025linear}. The idea is to write the entropy flux, $F_n^\ent$, in terms of a symmetric boundary matrix $\AMat$ that can be diagonalized by a transformation $\AMat = \TMat\, \LambMat\,\TMat^T$ so that
\begin{equation}
        F_n^{\ent}
   =
   \Uvec^T\AMat\Uvec
   =
   \Wvec^T\LambMat\Wvec,
   \label{eq:entFluxAlt}
\end{equation}
where $\Wvec = \TMat^T\Uvec$ defines a vector of (nonlinear) characteristic variables. 
\revOne{The variables $\Uvec$ come from the skew-symmetric analysis of the governing equation and may represent either primitive variables, as in the incompressible Euler equations, or Roe-type variables involving square roots, as in the shallow water equations \cite{nordstrom2022nonlinear}}.
The sign of the terms in the diagonal matrix $\LambMat$ specifies the minimal number of boundary conditions required to obtain a bound on the solution in terms of external data \cite{gustafsson1995time,nordstrom2022nonlinear, nordstrom2020, nordstrom2023nonlinear}. For linear problems, $\TMat$ is the matrix of eigenvectors and $\LambMat$ is the matrix of eigenvalues of the matrix $\AMat$. For nonlinear problems we will use a congruence transformation \revOne{of $\AMat$ to determine an appropriate transformation $\TMat$ that relates nonlinear and linear theory,} as discussed \revOne{in Section~\ref{sec:congruence}}.

The decomposition in \eqref{eq:entFluxAlt} exposes how solution information propagates in the directions determined by the sign of the terms in the diagonal matrix $\LambMat$, so we decompose $\LambMat = \LambPlus + \LambMinus$, where $\LambMat^\pm = \oneHalf\left(\LambMat \pm |\LambMat| \right)$ into matrices with only positive or negative entries. Possible zero entries are included in $\LambPlus$. We associate outgoing and incoming (nonlinear) characteristic
variables with those positive and negative terms by
\begin{equation}
   \label{eq:char_plus_minus}
   \Wplus = \IPlus\Wvec = \IPlus\TMat^T\Uvec \quad \textrm{and} \quad  \Wminus = \IMinus\Wvec = \IMinus\TMat^T\Uvec,
\end{equation}
where $\IPlus$ and $\IMinus$ are indicator matrices (incomplete identity matrices) that select components according to the signs of the terms in $\LambMat$. 

Again, following \cite{nordstrom2022nonlinear,nordstrom2023nonlinear,nordstrom2025open}, we prove the following purely algebraic result:
\begin{theorem}\label{thm:SATTheorem}
Let $\statevec G$ be a state vector
. Then
\begin{equation}
   \label{eq:SBP-energy}
 \Wvec^T\LambMat\Wvec + \Uvec^T\left(2\TMat\,\IMinus\sqrtLambM\left(\sqrtLambM\Wminus - \Gvec\right)\right)\ge -\statevec G^T\statevec G,
\end{equation}
\end{theorem}
\begin{proof}
Splitting in terms of $\Wvec^\pm$,
\begin{equation}
   \label{eq:generic_bound}
   \begin{aligned}
   \Wvec^T\LambMat\Wvec &+ \Uvec^T\left(2\TMat\,\IMinus\sqrtLambM\left(\sqrtLambM\Wminus - \Gvec\right)\right)\\[0.1cm]
   &= (\Wplus)^T\LambPlus\Wplus - (\Wminus)^T\absLambM\Wminus + 2\Uvec^T\TMat\,\IMinus\sqrtLambM\left(\sqrtLambM\Wminus - \Gvec\right)\\[0.1cm]
   &= (\Wplus)^T\LambPlus\Wplus - (\Wminus)^T\absLambM\Wminus + 2(\Wminus)^T\sqrtLambM\left(\sqrtLambM\Wminus - \Gvec\right).
   \end{aligned}
\end{equation}
Completing the square yields
\begin{equation}
   \label{eq:generic_bound2}
   \begin{aligned}
   (\Wplus)^T\LambPlus\Wplus &- (\Wminus)^T\absLambM\Wminus + 2(\Wminus)^T\sqrtLambM\left(\sqrtLambM\Wminus - \Gvec\right)
   \\[0.1cm]
   &= (\Wplus)^T\LambPlus\Wplus + (\Wminus)^T\absLambM\Wminus - 2 (\Wminus)^T\sqrtLambM\Gvec \pm \Gvec^T\Gvec\\[0.1cm]
   &= -\Gvec^T\Gvec + (\Wplus)^T\LambPlus\Wplus + (\Wminus)^T\absLambM\Wminus - 2 (\Wminus)^T\sqrtLambM\Gvec + \Gvec^T\Gvec\\[0.1cm]
   &= -\Gvec^T\Gvec + (\Wplus)^T\LambPlus\Wplus + \left(\sqrtLambM\Wminus - \Gvec\right)^{\!T}\!\left(\sqrtLambM\Wminus - \Gvec\right) \ge -\statevec G^T\statevec G,
   \end{aligned}
\end{equation}
which proves the result.
\end{proof}

\revOne{Now let $\statevec{G}$ consist of accurate boundary data in the form of the boundary operator $\sqrtLambM\Wminus$ such that $\sqrtLambM\Wminus - \Gvec$ is small, and interpret $2\TMat\,\IMinus\sqrtLambM\left(\sqrtLambM\Wminus - \Gvec\right)$ as the simultaneous approximation term (SAT) \cite{nordstrom2017roadmap}. Then}

\begin{theorem}
    \label{thm:FluxTheorem}
    If the numerical boundary flux function $\statevec{F}^*_n$ satisfies
    \begin{equation}
       \label{eq:generic_relation}
          \Vvec^T\left(\statevec{F}^*_n - \statevec{F}_n\right)
          =
          \Uvec^T\left(2\TMat\,\IMinus\sqrtLambM\left(\sqrtLambM\Wminus - \Gvec\right)\right),
    \end{equation}
    at inflow/outflow boundaries with external data specified as the vector $\statevec G$, then 
    \begin{equation}\label{eq:desiredResult}
        \frac{d\mathcal E}{dt}\le \int\limits_{\partial E,N} \statevec G^T\statevec{G}\hat s\dS
    \end{equation}
    so that the entropy rate is bounded solely by data.
\end{theorem}
\begin{proof}
    Starting from \eqref{eq:EntropyBoundaryIntegral}, we rewrite the entropy flux in the normal direction according to \eqref{eq:entFluxAlt} to have
    \begin{equation}
       \label{eq:ent_dgsem}
       \frac{d\energy}{dt} + \int\limits_{\partial E,N} \left\{\Wvec^T\LambMat\Wvec + \Vvec^T\left(\statevec{F}^*_n - \statevec{F}_n\right)\right\}\hat{s}\dS = 0.
    \end{equation}
    We compare the terms in \eqref{eq:ent_dgsem} with \eqref{eq:SBP-energy} and \eqref{eq:generic_relation} to arrive at the desired result \eqref{eq:desiredResult}.    
\end{proof}

\begin{remark}
   The statement \eqref{eq:generic_relation} is an algebraic relationship between the (possibly) vector valued numerical flux function $\statevec{F}^*_n$
   and the nonlinearly stable SAT from Thm.~\ref{thm:SATTheorem}.
   This is similar in spirit to creating an entropy conservative numerical flux function, see, e.g., \cite{Tadmor2003,winters2021}.
   Once a set of characteristic-type variables $\Wvec$ and the corresponding scaling matrix $\LambMat$ are identified, it is an algebraic exercise from \eqref{eq:generic_relation} to determine the components of the numerical flux function.
\end{remark}

\revOne{Theorems~\ref{thm:SATTheorem} and \ref{thm:FluxTheorem} illustrate that $\Gvec$ contains the external boundary data needed to balance the boundary operator $\sqrtLambM\Wminus$ and acquire a bound.
For complete details on the role of $\Gvec$ in the nonlinear SATs see~\cite{nordstrom2022nonlinear,nordstrom2023nonlinear,nordstrom2025open,nordstrom2025linear}.
For the Burgers equation and the shallow water equations specific forms of $\Gvec$ and how they relate to the numerical flux functions will be derived in Secs.~\ref{sec:burgers_first} and \ref{sec:swe_BCS}, respectively. 
}

To find a suitable numerical boundary flux function, one finds a transformation \eqref{eq:entFluxAlt}, where $\LambMat$ has the correct number of positive and negative entries to be consistent with linear theory. The condition \eqref{eq:generic_relation} uses two variable sets, $\statevec{V}$ and $\statevec{U}$,
so there is a potential variable mismatch for the contraction. If they are the same, then $\statevec{F}^*_n$ can be computed directly, as described for a one dimensional example in Sec.~\ref{sec:burgers}.
If they are different, we relate the variable sets $\statevec{V}$ and $\statevec{U}$ to each other, which we demonstrate for a two dimensional nonlinear system in Sec.~\ref{sec:swe}.
After the variable discrepancy is resolved, we show how to construct the desired numerical boundary flux function.
\revTwo{
\begin{remark}
   The relation \eqref{eq:generic_relation} yields a valid numerical flux provided there exists a common set of variables, $\Zvec$, between the entropy variables $\Vvec$ and the variables $\Uvec$.
   The three variable sets constitute different representations of the same physical state and are related by invertible transformations.
   It is always possible to find $\Vvec(\Zvec) = \mmatrix{M}_{\Vvec}(\Zvec)\Zvec$ and $\Uvec(\Zvec) = \mmatrix{M}_{\Uvec}(\Zvec)\Zvec$ to rewrite the variables in terms of the common variables $\Zvec$ via an outer product representation \cite{horn2012}, e.g.,
   \begin{equation}
      \mmatrix{M}_{\Vvec}(\Zvec) = \frac{\Vvec\Zvec^T}{\Zvec^T\Zvec}\quad\text{so that}\quad \mmatrix{M}_{\Vvec}(\Zvec)\Zvec = \frac{\Vvec\Zvec^T}{\Zvec^T\Zvec}\Zvec = \Vvec.
   \end{equation}
   Importantly, the matrices $\mmatrix{M}_{\Vvec}(\Zvec)$ and $ \mmatrix{M}_{\Uvec}(\Zvec)$ are not unique and their entries may depend nonlinearly on $\Zvec$.
   Note that these matrices are \textit{not} Jacobians; they describe the variable transformation itself, rather than its differential.
   Substituting the transformations, 
   \eqref{eq:generic_relation} becomes
   \begin{equation}
          (\mmatrix{M}_{\Vvec}\Zvec)^T\left(\statevec{F}^*_n - \statevec{F}_n\right)
          =
          (\mmatrix{M}_{\Uvec}\Zvec)^T\left(2\TMat\,\IMinus\sqrtLambM\left(\sqrtLambM\Wminus - \Gvec\right)\right).
   \end{equation}
   Manipulating this expression we find the provably stable numerical boundary flux function takes the form
   \begin{equation}\label{eq:unikNumFlux}
      \statevec{F}^*_n= 2\mmatrix{M}_{\Vvec}^{-T}\mmatrix{M}_{\Uvec}^T\,\TMat\,\IMinus\sqrtLambM\left(\sqrtLambM\Wminus - \Gvec\right) + \statevec{F}_n,
   \end{equation}
   provided that the characteristic-type variables $\Wvec$ and diagonal matrix $\LambMat$ can be found.\\
\end{remark}
}
\revTwo{In Sec.~\ref{sec:swe}, we use the primitive variables as a common variable set $\Zvec$ for the two-dimensional shallow water equations.
We then choose a diagonal matrix $\LambMat$ and describe how to create a transformation of the boundary matrix to reveal a set of characteristic-type variables.
Then, nonlinear stable boundary fluxes can be determined.
To summarize, once a choice for a common variable set is made, with corresponding transformations $\mmatrix{M}_{\Vvec}(\Zvec)$ and $\mmatrix{M}_{\Uvec}(\Zvec)$, and a suitable transformation matrix $\TMat$ in the decomposition of the boundary matrix is created for a chosen diagonal matrix $\LambMat$, then the numerical flux is uniquely determined from \eqref{eq:unikNumFlux}.
}

 
\section{A stable nonlinear inflow boundary flux for the Burgers equation}\label{sec:burgers}

To illustrate the process and to motivate and highlight the vital steps, we demonstrate the use of Thm.~\ref{thm:FluxTheorem} by deriving an energy stable boundary flux for the Burgers equation. \revTwo{Although nonlinear, the Burgers equation is sufficiently simple that the boundary flux function in \eqref{eq:EntropyBoundaryIntegral} can be derived via the energy method and a nonstandard rewriting of the boundary terms.}
We show here that the derivation \revTwo{using the energy method} gives the same result as Thm.~\ref{thm:FluxTheorem}. 

\subsection{\revTwo{Boundary flux derivation for Burgers' equation}}\label{sec:burgers_first}
In one space dimension, the Burgers equation is
\begin{equation}
u_t + f_x  = 0 \quad x\in \Omega = [x_L, x_R]
\label{eq:BasicBurgers}
\end{equation}
where the flux is $f(u) = \frac{1}{2}u^2 $. 
The entropy, $\energy = \oneHalf u^2$, is also the energy for the Burgers equation, which allows us to prove energy stability.
The entropy variable $\partial \energy/\partial u = u$, which is also the state variable, leads to an entropy flux $f^\ent = \frac{1}{3}u^3$.

We first seek boundary conditions for the PDE that ensure that the entropy is bounded by data. It is sufficient for this example to assume $u > 0$ and only consider the boundary condition at $x_{L}$. Multiplying by the entropy variable $u$ and integrating over the domain, the total entropy/energy satisfies 
\begin{equation}
\energy_t = -\left. \frac{1}{3}u^3\right|_{x_R} +  \left. \frac{1}{3}u^3\right|_{x_L}.
\end{equation}
 When $u > 0$, $\left. \frac{1}{3}u^3\right|_{x_R} > 0$ and has the correct sign to dissipate energy, so no boundary condition is needed at $x_R$. At the left boundary, if we ensure that the contribution is a square of specified data, i.e., if the contribution from the left (inflow) side is equal to some $G^{2}$, where $G(t)$ is a function solely of specified data, then the energy is bounded exclusively by data, i.e.,
$\energy_{t} \le G^{2}$ and $\energy(T) \le \int_{0}^{T} G^{2}dt$.

To get the square and ensure boundedness, then, we split the boundary term as
\begin{equation}
\frac{1}{3}u^3=u \left(\frac{1}{3}|u|\right) u = u\left(\sqrt{ \frac{1}{3}|u| }\right)\left(\sqrt{ \frac{1}{3}|u| }\right) u = \left( \sqrt{ \frac{1}{3}|u| } u \right)^{2} \equiv \overline W ^{2}.
\end{equation}
In contrast to the characteristic variable $W=u$, $\overline W$ includes the wave speed.
Therefore, if we set the boundary condition at the left as
\begin{equation}
G = \overline W(\uext) = \sqrt{ \frac{1}{3}|\uext | }\uext,
\label{eq:StrongBurgersCondition}
\end{equation}
then $BT \le G^{2}$ where $G$ is a function only of boundary data. This means that, unlike in the linear analysis, one does not specify the solution at the boundary, but \eqref{eq:StrongBurgersCondition} instead, which includes the wave speed.


Since \eqref{eq:EntropyBoundaryIntegral} imposes the boundary conditions in weak form as a penalty, we need the appropriate penalty that ensures that $\mathcal E$ is bounded by data in the same way. From \eqref{eq:StrongBurgersCondition}, we penalize the PDE by adding a term proportional to $\overline W - G$, where $G$ is defined in \eqref{eq:StrongBurgersCondition},
\begin{equation}
u_t + f_x + \left.\sigma l(x - x_L)\left(\overline W-G\right)\right|_{x_L} = 0,
\label{eq:BurgStrongPenalty}
\end{equation}
and $l(x)$ is a lifting operator defined as
$\int_{x_L}^{x_R} \phi l(\psi)dx = \phi\psi|_{x_L}$ \cite{arnold2002unified,sudirham2003dgfem}.
The penalty is on the characteristic variable $\overline W$, which, again, differs from linear problems where the dependent  variable $u$ is penalized towards the external solution, $\uext$.

We must then find $\sigma$ so that the entropy is bounded only by $G$.
When we multiply by the entropy variable $u$ and integrate over the domain, the total entropy satisfies
\begin{equation}
\energy_t +  \left. f^\ent\right|_{x_L}^{x_R}  +  \left.\sigma u\left(\overline W-G\right)\right|_{x_L} = 0.
\label{eq:EntropyWithPenalty}
\end{equation}
As before, we take $u > 0$ so that we can ignore the boundary on the right and specify the data on the left. At the left boundary, we must bound 
\begin{equation}
\label{eq:sigmaEq}
\begin{split}
\energy_{t} \le BT 
\equiv \left. f^{\ent} +\sigma u\left(\overline W-G\right) \right|_{x_L}
= \frac{1}{3}u^{3}+\sigma u\left(\overline W-G\right)
&=  \left( \sqrt{ \frac{1}{3}|u| } u\right)^{2} +\sigma u \overline W-\sigma u G
\\[0.1cm]
&= \overline W^2 + \sigma u \overline W- \sigma u G.
\end{split}
\end{equation}
We can bound $BT \le G^2$ if we choose 
$\sigma u = \left(-2\sqrt{ \frac{1}{3}|u| }\right)u = -2\overline W$ and complete the square, giving
\begin{equation}
BT = -\overline W^{2} +2\overline W G -G^2 + G^2 =- (\overline W- G)^2 + G^2 \le G^2.
\end{equation}
Note, this choice of $\sigma$ is analogous to the generic penalty statement in Thm.~\ref{thm:SATTheorem}.


Finally, we turn to the construction of the numerical boundary flux, $F^*$.
When we compare the boundary term in the formulation \eqref{eq:EntropyBoundaryIntegral} with \eqref{eq:EntropyWithPenalty}, the bounds match if
\begin{equation}
    \label{eq:BurgersBT}
    BT = F^{\ent} + u (F^{*}-F) = F^{\ent} +\sigma u\left(\overline W-G\right).
\end{equation}
That is, we need $F^{*}$ such that
\begin{equation}
    F^* = \sigma(\overline W-G) + F.
\end{equation}
When we substitute for $\overline W$, $G$, $F$, and $\sigma$, use the fact that $u > 0$, and rearrange, we find the numerical boundary flux function that satisfies Thm. \ref{thm:FluxTheorem} and ensures that the entropy is bounded only by data to be
\begin{equation}
    \label{eq:num_burgers}
    F^*(\uext,u) = \frac{1}{3}\left( 2 \uext\sqrt{|\uext| |u|} -\oneHalf u^{2}  \right).
\end{equation}

\begin{remark}
    The numerical boundary flux uses the geometric mean of the interior solution and the data.
\end{remark}
\begin{remark}
If one uses the entropy conservative two-point flux $(\uext^{2} + u \uext + u^{2})/6$ \cite{gassner2013skew} as the numerical boundary flux in \eqref{eq:BurgersBT}, then with $u > 0$
\begin{equation}
BT = F^{\ent} + u (F^{*}-F) = u \left(F^{*}-\frac{1}{3}F\right) = u \left( \frac{\uext^{2} + u \uext + u^{2}}{6} - \frac{u^{2}}{6}  \right) = \frac{1}{6}\left( u \uext^{2}  + u^{2}\uext\right).
\label{eq:BTWithFecFlux}
\end{equation}
This boundary term is not a square bounded by data independent of $u$, and hence does not lead to an estimate solely in terms of data.
If, instead, we use the local Lax-Friedrichs (LLF) flux \cite{toro2013riemann} at the boundary, then
    \begin{equation}\label{eq:brokenLLF}
     BT = \left\{
     \begin{aligned}
        &-\frac{u^3}{6} + \frac{u \uext^2}{2} - \frac{u(u - \uext)^2}{4}, \quad u >  \uext\\[0.2cm]
        &-\frac{u^3}{6} + \frac{u \uext^2}{2} + \frac{u(u - \uext)^2}{4}, \quad u < \uext.
    \end{aligned}
    \right.
    \end{equation}
    Again, the boundary term is not bounded independently of the interior state, $u$, and using the flux \eqref{eq:brokenLLF} does not lead to an entropy bound.
   \label{rem:ECBCFlux}
\end{remark}

\subsection{A derivation using the general theory}\label{sec:burgers_general}
The generic statement \eqref{eq:generic_relation} yields the same numerical boundary flux, provided we can form the (nonlinear) characteristic variables $\Wvec$ by finding the matrix $\TMat$, which then guides the ansatz for boundary data function $\Gvec$.
As before, we look only at the $u > 0$ case. At the left boundary for the scalar equation, $\IMinus \gets 1$.
For the scalar Burgers equation, $\Uvec = \Vvec \gets u$, i.e., the conservative and entropy variables are the same,
so the transformation $\TMat \gets 1$.
Furthermore, the boundary data according to the continuous analysis \eqref{eq:StrongBurgersCondition} is
\begin{equation}
    \Gvec \leftarrow G = \sqrt{ \frac{1}{3}|\uext|} \uext,
    \label{eq:Grepacement}
\end{equation}
together with the characteristic variable and scaling factor,
\begin{equation}
    \Wminus \leftarrow u \quad\text{and}\quad \sqrtLambM \gets \sqrt{\frac{1}{3}|u|}.
    \label{eq:Wreplacement}
\end{equation}
Substituting \eqref{eq:Grepacement} and \eqref{eq:Wreplacement} into \eqref{eq:generic_relation} we find 
\begin{equation}
    -u (f^* - f) = u\left(2\left(\sqrt{\frac{1}{3}|u|}\right)\left(\sqrt{\frac{1}{3}|u|}u -  \sqrt{ \frac{1}{3}|\uext | }\uext\right)\right)
    =u \left(\frac{2}{3}u^2 - \frac{2}{3}\sqrt{|u||\uext|}\uext\right),
\end{equation}
where the negative sign on the left comes from the one dimensional outward pointing normal direction at the left boundary.
Solving for the numerical boundary flux function and recalling that the Burgers flux is $f = u^2/2$, we find
\begin{equation}
    \label{eq:alt_flux_deriv}
    f^*(\uext, u) = \frac{2}{3}\sqrt{|u||\uext|}\uext - \frac{2}{3}u^2 + f
    = \frac{1}{3}\left(2\sqrt{|u||\uext|}\uext - \frac{1}{2}u^2\right).
\end{equation}

Thus, the formalism of \eqref{eq:generic_relation} leads to a numerical boundary flux function \eqref{eq:alt_flux_deriv} that is identical to the one derived from \revTwo{the energy method} in \eqref{eq:num_burgers}.
Note that the creation of the boundary flux function from \eqref{eq:generic_relation} requires knowledge of the characteristic variables $\Wvec$, the wave speeds $\LambMat$, and the boundary terms $\Gvec$.

\revOne{
\begin{remark}\label{rem:Grole}
The final form of the numerical flux derived via the energy method \eqref{eq:num_burgers} or from the generic result of Thm.~\ref{thm:FluxTheorem} in \eqref{eq:alt_flux_deriv} is a function of the external solution state $\uext$.
It does not explicitly contain the boundary data function $\Gvec$, but is derived from it, as shown in \eqref{eq:StrongBurgersCondition} and \eqref{eq:Grepacement}.
\end{remark}
}

\section{Stable nonlinear boundary fluxes for the shallow water equations}\label{sec:swe}

We leverage knowledge gained from the Burgers analysis to create numerical boundary flux functions for 
the shallow water equations in two spatial dimensions with a flat bottom topography.
The governing equations are written in conservative form 
\begin{equation}
    \label{eq:swe_cons}
    \statevec{q}_t + (\statevec{f}_1)_x + (\statevec{f}_2)_y = 0, \quad x\in \Omega,
\end{equation}
with the conservative variables $\statevec{q} = (h, hv_1, hv_2)^T$ and fluxes
\begin{equation}
    \statevec{f}_1
    =
    \begin{pmatrix}
        h v_1\\
        h v_1^2 + \frac{g}{2}h^2\\
        h v_1 v_2
    \end{pmatrix},
    \quad
    \statevec{f}_2
    =
    \begin{pmatrix}
        h v_2\\
        h v_1 v_2\\
        h v_2^2 + \frac{g}{2}h^2
    \end{pmatrix},
\end{equation}
where $h$ is the water height, $g$ is now the gravitational constant,
and $v_1, v_2$ are the velocities in the $x$ and $y$ directions~\cite{LeVeque2002}.
The normal flux at the boundary $\partial\Omega$ with outward normal vector $\spacevec{n} = (n_1, n_2)^T$ is
\begin{equation}
   \label{eq:normal_flux}
   \statevec{f}_n = n_1 \statevec{f}_1 + n_2 \statevec{f}_2
   =
   \begin{pmatrix}
   h v_n\\[0.05cm]
   h v_1 v_n  + \frac{g}{2}h^2 n_1\\[0.05cm]
   h v_2 v_n + \frac{g}{2}h^2 n_2\\[0.05cm]
   \end{pmatrix},
\end{equation}
with the normal velocity $v_n = n_1 v_1 + n_2 v_2$.

The relation \eqref{eq:generic_relation} requires several variable sets and matrices to determine a numerical boundary flux function.
We start from the mathematical entropy analysis of the shallow water equations where the specific total energy plays the role of the entropy function \cite{fjordholm2011well,wintermeyer2017entropy}
\begin{equation}
   \label{eq:ent_func}
   S(\statevec{q}) = \frac{h}{2}(v_1^2 + v_2^2) + \frac{g}{2}h^2.
\end{equation}
The associated entropy flux in the normal direction, found via a compatibility condition \cite{fjordholm2011well,wintermeyer2017entropy}, is
\begin{equation}
   \label{eq:ent_fluxes2}
   F^{\ent}_n= \frac{hv_n}{2}(v_1^2+v_2^2) + g h^2 v_n.
\end{equation}
To contract a conservation law from physical space into entropy space one uses the entropy variables~\cite{tadmor1984}
\begin{equation}
   \label{eq:ent_vars}
   \Vvec = \frac{\partial S}{\partial \statevec{q}}
   =
   \begin{pmatrix}
   g h - \frac{1}{2}(v_1^2 + v_2^2)\\[0.05cm]
   v_1\\[0.05cm]
   v_2\\[0.05cm]
   \end{pmatrix}.
\end{equation}
It is precisely these entropy variables $\Vvec$ that are needed on the left-hand-side of \eqref{eq:generic_relation}.

Analogous to the nonlinearly stable boundary analysis of \cite{nordstrom2022nonlinear,nordstrom2023nonlinear}, we choose the evolution variables $\Uvec$ in \eqref{eq:generic_relation} to be the scaled and rotated primitive variables, $\Uprimvec$, defined as
\begin{equation}
    \label{eq:rotatedPrimVars}
    \Uvec
    =
   \frac{1}{\sqrt{2g}}
   \begin{pmatrix}
   g h\\[0.05cm]
   \sqrt{gh} v_n\\[0.05cm]
   \sqrt{gh} v_\tau\\[0.05cm]
   \end{pmatrix}
    =
   \frac{1}{\sqrt{2g}}
   \begin{pmatrix}
   gh\\[0.05cm]
   c v_n\\[0.05cm]
   c v_\tau\\[0.05cm]
   \end{pmatrix}
   =
   \frac{1}{\sqrt{2g}}
   \begin{pmatrix}
   g & 0 & 0 \\[0.05cm]
   0 & c & 0\\[0.05cm]
   0 & 0 & c \\[0.05cm]
   \end{pmatrix}
   \begin{pmatrix}
   1 & 0 & 0\\[0.05cm]
   0 & n_1 & n_2\\[0.05cm]
   0 & -n_2 & n_1\\[0.05cm]
   \end{pmatrix}      
   \begin{pmatrix}
   h\\[0.05cm]
   v_1\\[0.05cm]
   v_2\\[0.05cm]
   \end{pmatrix}
   =
   \SMat\,\NMat\Uprimvec,
\end{equation}
where $c = \sqrt{gh}$ is the wave speed, $v_\tau = -n_2 v_1 + n_1 v_2$ is the tangential velocity, $\SMat$ is a diagonal scaling matrix, and $\NMat$ is a normal rotation matrix.
Unlike for the Burgers analysis in Sec.~\ref{sec:burgers}, there is a mismatch between the entropy variables, $\Vvec$, and evolution variables, $\Uvec$.
To relate the entropy variables to the primitive variables for the shallow water equations we use the transformation
\begin{equation}
   \label{eq:ent_transform}
   \Vvec
   =
   \begin{pmatrix}
   g h - \frac{1}{2}(v_1^2 + v_2^2)\\[0.05cm]
   v_1\\[0.05cm]
   v_2\\[0.05cm]   
   \end{pmatrix}
   =
   \begin{pmatrix}
   g & -\frac{v_1}{2} & -\frac{v_{2}}{2}\\[0.05cm]
   0 & 1 & 0\\[0.05cm]
   0 &  0 & 1\\[0.05cm]
   \end{pmatrix}   
   \begin{pmatrix}
   h\\[0.05cm]
   v_1\\[0.05cm]
   v_2\\[0.05cm]
   \end{pmatrix}
   =
   \MMat\Uprimvec.
\end{equation}

From the generic relation in Thm. \ref{thm:FluxTheorem} we are equipped to describe the specific relation
required for the two-dimensional shallow water equations.
Substituting \eqref{eq:rotatedPrimVars} as well as the relationship between entropy variables and primitive variables \eqref{eq:ent_transform} we have
\begin{equation}
   \label{eq:swe_relation1}
   \Vvec^T\left(\statevec{F}^*_n - \statevec{F}_n\right) = (\MMat\Uprimvec)^T\left(\statevec{F}^*_n - \statevec{F}_n\right)
   = \Uprimvec^T\left(\MMat^T\left(\statevec{F}^*_n - \statevec{F}_n\right)\right)
\end{equation}
To obtain \eqref{eq:generic_relation} we require that \eqref{eq:swe_relation1} is equal to
\begin{equation}
   \label{eq:swe_relationIntermediate}
   \Uvec^T\left(2\TMat\,\IMinus\sqrtLambM\left(\sqrtLambM\Wminus - \Gvec\right)\right)
   =
   \Uprimvec^T\NMat^T\SMat\left(2\TMat\,\IMinus\sqrtLambM\left(\sqrtLambM\Wminus - \Gvec\right)\right),
\end{equation}
where we have inserted \eqref{eq:rotatedPrimVars}.
Now \eqref{eq:swe_relation1} and \eqref{eq:swe_relationIntermediate} are contracted with the primitive variables, $\Uprimvec$, and \eqref{eq:generic_relation} becomes
\begin{equation}
   \label{eq:swe_relation2}
   \left(\statevec{F}^*_n - \statevec{F}_n\right)= 2\invMMatT\NMat^T\SMat\,\TMat\,\IMinus\sqrtLambM\left(\sqrtLambM\Wminus - \Gvec\right).
\end{equation}
%

The remaining piece in the analysis is to find a transformation matrix $\TMat$.
This crucial transformation defines the (nonlinear) characteristic variables, $\Wvec$, and the diagonal matrix $\LambMat$.
The characteristic variables $\Wvec$ are determined by rewriting the entropy flux, a scalar quantity, in the quadratic form \eqref{eq:entFluxAlt}.
There is freedom in the choice of the boundary matrix $\AMat$ that one should exploit in this process.
Following \cite{nordstrom2022nonlinear}, there is a one-parameter family of skew-symmetric forms of the shallow water equations with the boundary matrix
\begin{equation}
    \label{eq:AMatrix1}
    \AMat
    =
    \begin{pmatrix}
        2\beta v_n & (1-\beta) c & 0\\[0.1cm]
        (1-\beta) c & v_n & 0\\[0.1cm]
        0 & 0 & v_n\\[0.1cm]
    \end{pmatrix},
    \quad \beta\in\mathbb{R},
\end{equation}
where $\beta$ is a free parameter.
Contracting the boundary matrix \eqref{eq:AMatrix1} from the left and right with $\Uvec$ \eqref{eq:rotatedPrimVars}, it is straightforward to find
\begin{equation}
    \begin{aligned}
    \Uvec^T\AMat\Uvec
    = 
    \frac{1}{2g}\left(c^2 v_n^3 + c^2 v_n v_\tau^2 + 2 c^4 v_n (1-\beta) + 2 c^4 v_n \beta \right) 
    &= 
    \frac{h v_n}{2g}\left( v_n^2 + v_\tau^2 \right) + g h^2 v_n\\[0.1cm]
    &=
    \frac{h v_n}{2g}\left( v_1^2 + v_2^2 \right) + g h^2 v_n 
    =
    F_n^\ent.
    \end{aligned}
\end{equation}
All terms involving $\beta$ cancel under contraction,
so there is no guidance at this point for how to select the parameter in the skew-symmetric formulation.

\subsection{Congruence transformation}\label{sec:congruence}
To determine the free parameter, $\beta$, we use the fact that congruent matrices have the same number of positive, negative and zero eigenvalues, by way of Sylvester's law of inertia \cite{horn2012}.
Matrix congruency preserves quadratic forms \cite{horn2012}, so that $\Wvec^T\LambMat\Wvec$ recovers the entropy flux in the normal direction $F_n^{\ent}$ independent of the choice of $\Wvec$.
Congruency also means that the signature of the matrix $\AMat$ and $\LambMat$ are identical under the transformation
\begin{equation}
    \Uvec^T\AMat\Uvec = \Uvec^T\TMat\,\LambMat\,\TMat^T\Uvec = \Wvec^T\LambMat\,\Wvec,\quad \text{where} \quad \Wvec = \TMat^T\Uvec.
\end{equation}
The number of positive and negative entries in the diagonal matrix $\LambMat$ determine how many boundary conditions that are required to guarantee that the solution is bounded by data~\cite{nordstrom2020}, as in Sec.~\ref{sec:generic}.

To derive the congruence transformation $\AMat = \TMat\,\LambMat\,\TMat^T$
we choose the target diagonal matrix $\LambMat$ to be
\begin{equation}
    \label{eq:switch-matrix}
    \LambMat = \text{diag}(v_n - c , v_n, v_n + c),
\end{equation}
which is the diagonal matrix of eigenvalues for the shallow water flux Jacobian \cite{LeVeque2002}.
We select this particular diagonal matrix so that the number of positive and negative entries change when the flow regime transitions from subcritical (fluvial) to supercritical (torrential) determined by the sign and magnitude of the normal Froude number $v_n / c$.
This choice also ensures that the placement and number of boundary conditions for the nonlinear problem is consistent with the linear analysis of the shallow water equations \cite{oliger1978,shallowwaterbook}.

We introduce an additional parameter $\alpha$ where
\begin{equation}
    \label{eq:alphaEqn}
    \alpha^2 = 2\beta \quad\text{and}\quad \alpha = 1 - \beta
\end{equation}
to match terms in \eqref{eq:AMatrix1}.
From these two conditions we find that
\begin{equation}
    \label{eq:alphas}
    \frac{\alpha^2}{2} + \alpha = 1
    \quad\text{or}\quad
    \alpha = -1 \pm \sqrt{3},
\end{equation}
and $\beta = 1 - \alpha = 2 \mp \sqrt{3}$.
Of the two choices in \eqref{eq:alphas} we select the positive value $\alpha = -1 + \sqrt{3}$, as the negative value yields a wider spectral radius for the matrix $\AMat$.
A further physical motivation for this choice will be given in Sec.~\ref{sec:subcrit-outflow}.
The goal is then to find a transformation matrix $\TMat$ so that
\begin{equation}
    \label{eq:congruence}
    \TMat\,\LambMat\,\TMat^T
    =
    \TMat
    \begin{pmatrix}
        v_n - c& 0 & 0\\[0.1cm]
        0 & v_n & 0\\[0.1cm]
        0 & 0 & v_n + c\\[0.1cm]
    \end{pmatrix}    
    \TMat^T
    =
    \begin{pmatrix}
        \alpha^2 v_n & \alpha c & 0\\[0.1cm]
        \alpha c & v_n & 0\\[0.1cm]
        0 & 0 & v_n\\[0.1cm]
    \end{pmatrix}
    =
    \AMat.
\end{equation}
After many algebraic manipulations found in \ref{app:congrunet}, we derive the desired transformation matrix
\begin{equation}
    \label{eq:TMat}
    \TMat
    =
    \begin{pmatrix}
        \frac{\alpha}{\sqrt{2}} & 0 & \frac{\alpha}{\sqrt{2}}\\[0.1cm]
        -\frac{1}{\sqrt{2}} & 0 & \frac{1}{\sqrt{2}}\\[0.1cm]
        0 & 1 & 0
    \end{pmatrix}.
\end{equation}
%
\begin{remark}
    The matrices $\TMat$ \eqref{eq:TMat} and $\LambMat$ \eqref{eq:switch-matrix} \textbf{are not} the eigenvectors and eigenvalues, respectively, of the matrix $\AMat$, as would be the case for a similarity transformation. 
    Instead, the invertible matrix $\TMat$ in the congruence relation $\AMat = \TMat\,\LambMat\,\TMat^T$ is created to \textit{simultaneously} preserve the quadratic form that relates the boundary matrix $\AMat$ and the entropy flux $F_n^\ent$ as well as the signature of the diagonal eigenvalue matrix of the shallow water flux Jacobian \eqref{eq:switch-matrix}.
\end{remark}
\noindent From the transformation matrix $\TMat$, we define the set of (nonlinear) characteristic variables
\begin{equation}
    \label{eq:swe-charvars}
    \Wvec = \TMat^T \Uvec
    =
    \frac{c}{2\sqrt{g}}
    \begin{pmatrix}
    \alpha c - v_n\\[0.1cm]
    \sqrt{2} v_\tau\\[0.1cm]
    \alpha c + v_n\\[0.1cm]    
    \end{pmatrix}.
\end{equation}

\subsection{Nonlinear numerical boundary flux functions for the shallow water equations}\label{sec:swe_BCS}
All the necessary parts in the relation \eqref{eq:swe_relation2} are now established and
we are ready to derive numerical boundary flux functions $\statevec{F}^*_n$, which will change for different flow regimes. 
The signs in the diagonal matrix \eqref{eq:switch-matrix} progress through four states depending on the direction and magnitude of the normal velocity:
\begin{itemize}
    \item Supercritical outflow; $v_n > 0$ and $|v_n| > c$ has zero negative values so $\statevec F^*$ uses no boundary data
    \item Subcritical outflow; $v_n > 0$ and $|v_n| < c$ has one negative value so $\statevec F^*$ uses one boundary data value
    \item Subcritical inflow; $v_n < 0$ and $|v_n| < c$ has two negative values so $\statevec F^*$ uses two boundary data values
    \item Supercritical inflow; $v_n < 0$ and $|v_n| > c$ has three negative values so $\statevec F^*$ uses three boundary data values
\end{itemize}

\revOne{Again, the numerical boundary fluxes $\statevec{F}^*_n$ are functions of the internal and external solution states and will not explicitly contain the boundary term $\Gvec$, see Remark~\ref{rem:Grole}.}

\subsubsection{Supercritical outflow numerical boundary flux}\label{sec:supercrit-outflow}

For shallow water supercritical outflow, $v_n > 0$ and $v_n > c$ where $c = \sqrt{gh}$.
From the diagonal matrix \eqref{eq:switch-matrix}
there are no negative terms so that
\begin{equation}
\absLambM = \text{diag}(0,\,0,\,0),
\end{equation}
and the indicator matrix is $\IMinus = \text{diag}(0,\,0,\,0)$.
This means that there are no incoming characteristics at a supercritical outflow boundary. 
Thus, no penalty terms are required to obtain a bound on the nonlinear solution in terms of data.
It is a boundary where all the solution information comes from inside of the domain \cite{toro2013riemann,bristeau2001boundary}.
To recover this action in the split form DGSEM we choose
\begin{equation}
    \label{eq:supercrit-outflow}
    \statevec{F}^*_n(\Uvec, \Uext)
    =
    \statevec{F}_n(\Uvec),
\end{equation}
so that the penalty term at the boundary in \eqref{eq:ent_dgsem} cancels.
In practice, any \textit{consistent} numerical flux function can be used to create \eqref{eq:supercrit-outflow} by choosing $\Uext = \Uvec$ since consistency means that
\begin{equation}
    \statevec{F}^*_n(\Uvec, \Uext) = \statevec{F}^*_n(\Uvec, \Uvec) = \statevec{F}_n(\Uvec).
\end{equation}

\subsubsection{Subcritical outflow numerical boundary flux}\label{sec:subcrit-outflow}

For shallow water subcritical outflow, $v_n > 0$ and $v_n < c$.
From the diagonal matrix \eqref{eq:switch-matrix}
there is one negative term so that
\begin{equation}
\absLambM = \text{diag}(|v_n - c|,\,0,\,0),
\end{equation}
and $\IMinus = \text{diag}(1,\,0,\,0)$.
The incoming characteristic variable at a subcritical outflow boundary is therefore
\begin{equation}
    \Wminus = \frac{c}{2\sqrt{g}}\begin{pmatrix}
        \alpha c - v_n\\[0.1cm]
        0\\[0.1cm]
        0
    \end{pmatrix}.
\end{equation}
Analogous to the Burgers analysis, we weakly impose the data for the boundary vector
\begin{equation}
    \label{eq:GvecAnsatz}
    \Gvec = \sqrt{\absLambMExt}\WminusExt\quad\text{with}\quad \WminusExt = \IMinus\TMat\Uext,
\end{equation}
in terms of the external rotated and scaled primitive variables \eqref{eq:rotatedPrimVars}
\begin{equation}
    \label{eq:primVarsAnsatz}
    \Uext
    =
   \frac{1}{\sqrt{2g}}
   \begin{pmatrix}
   g & 0 & 0 \\[0.05cm]
   0 & \cext & 0\\[0.05cm]
   0 & 0 & \cext \\[0.05cm]
   \end{pmatrix}
   \begin{pmatrix}
   \hext\\[0.05cm]
   \vnext\\[0.05cm]
   \vtext\\[0.05cm]
   \end{pmatrix},
\end{equation}
and
\begin{equation}
    \absLambMExt = \text{diag}(|\vnext - \cext|,\, 0,\, 0).
\end{equation}

We find in \ref{app:subcrit-outflow} that the numerical boundary flux function for subcritical outflow that satisfies \eqref{eq:swe_relation2} is
\begin{equation}
    \label{eq:subcrit-outflow}
    \statevec{F}^*_n(\Uvec, \Uext)
      =
      \begin{pmatrix}
      \frac{\alpha}{2} h v_n + (1-\alpha)h c + \frac{\alpha}{2 g}c v_n^2 - \frac{\alpha}{2g} \geo{\lambda_1} \cext (\alpha \cext - \vnext)\\\hdashline
      \left(\frac{\alpha}{4} + \frac{1}{2}\right) h v_1 v_n + \frac{1 - \alpha}{2} h c v_1 + \frac{\alpha}{4g} c v_1 v_n^2
         + (1-\alpha)\frac{g h^2}{2} n_1 + \frac{h v_n}{2} ((1+\alpha)c - v_n) n_1\\[0.1cm] 
         - \frac{1}{4g} \geo{\lambda_1} \cext (\alpha \cext - \vnext)(\alpha v_1 - 2 c n_1)\\\hdashline
      \left(\frac{\alpha}{4} + \frac{1}{2}\right) h v_2 v_n + \frac{1 - \alpha}{2} h c v_2 + \frac{\alpha}{4g} c v_2 v_n^2
         + (1-\alpha)\frac{g h^2}{2} n_2 + \frac{h v_n}{2} ((1+\alpha)c - v_n) n_2\\[0.1cm] 
         - \frac{1}{4g} \geo{\lambda_1} \cext (\alpha \cext - \vnext)(\alpha v_2 - 2 c n_2)
      \end{pmatrix}
\end{equation}
with $\alpha = -1 + \sqrt{3}$. The auxiliary variables and geometric means are
\begin{equation}
    \begin{aligned}
        c = \sqrt{gh}, \quad 
        \cext = \sqrt{g\hext}, \quad
        \vnext = n_1 v_1^\ext + n_2 v_2^\ext,
        \\ 
        \geo{h} = \sqrt{h\hext}, \quad
        \geo{\lambda_1} = \sqrt{(c - v_n) (\cext -\vnext)}.
    \end{aligned}
\end{equation}
The numerical boundary flux \eqref{eq:subcrit-outflow} is consistent, as can be shown by taking $\Uext = \Uvec$, recalling that $\alpha^2 + 2 \alpha - 2 = 0$.

\begin{remark}
    Interestingly, the last term in each of the momentum flux components in \eqref{eq:subcrit-outflow} resembles a Riemann invariant for the shallow water equations \cite{LeVeque2002}; however the velocity is scaled by $\alpha$.
    To not switch the sign of the Riemann invariant type term, we select the positive value $\alpha$ from \eqref{eq:alphas}.
\end{remark}

\subsubsection{Subcritical inflow numerical boundary flux}\label{sec:subcrit-inflow}

At boundaries of subcritical inflow, $v_n < 0$ and $|v_n| < c$.
From the diagonal matrix \eqref{eq:switch-matrix} there are two negative entries so that
\begin{equation}
\absLambM = \text{diag}(|v_n - c|,\,|v_n|,\,0),
\end{equation}
and $\IMinus = \text{diag}(1,\,1,\,0)$.
The incoming characteristic variables at a subcritical inflow boundary become
\begin{equation}
    \Wminus = \frac{c}{2\sqrt{g}}\begin{pmatrix}
        \alpha c - v_n\\\sqrt{2}v_\tau\\0
    \end{pmatrix}.
\end{equation}
We weakly impose the data with the boundary vector ansatz \eqref{eq:GvecAnsatz}, external evolution variables \eqref{eq:primVarsAnsatz}, and
\begin{equation}
    \absLambMExt = \text{diag}(|\vnext - \cext|,\, |\vnext|,\, 0)
\end{equation}
and find in \ref{app:subcrit-inflow} that the numerical boundary flux function for subcritical outflow that satisfies \eqref{eq:swe_relation2} is
\begin{equation}
    \label{eq:subcrit-inflow}
      \statevec{F}^*_n(\Uvec, \Uext)
      =
      \begin{pmatrix}
      \frac{\alpha}{2} h v_n + (1-\alpha)h c + \frac{\alpha}{2 g} c v_n^2 - \frac{\alpha}{2g} \geo{\lambda_1} \cext (\alpha \cext - \vnext)\\\hdashline
      \left(\frac{\alpha}{4} - \frac{1}{2}\right) h v_1 v_n + \frac{1 - \alpha}{2} h c v_1 + \frac{\alpha}{4g} c v_1 v_n^2
         + (1-\alpha)\frac{g h^2}{2} n_1 + \frac{h v_n}{2} ((1+\alpha)c + v_n) n_1\\[0.1cm] 
         - \frac{1}{4g} \geo{\lambda_1} \cext (\alpha \cext - \vnext)(\alpha v_1 - 2 c n_1) + \geo{\lambda_2} \geo{h} \vtext n_2\\\hdashline
      \left(\frac{\alpha}{4} - \frac{1}{2}\right) h v_2 v_n + \frac{1 - \alpha}{2} h c v_2 + \frac{\alpha}{4g} c v_2 v_n^2
         + (1-\alpha)\frac{g h^2}{2} n_2 + \frac{h v_n}{2} ((1+\alpha)c + v_n) n_2\\[0.1cm] 
         - \frac{1}{4g} \geo{\lambda_1} \cext (\alpha \cext - \vnext)(\alpha v_2 - 2 c n_2) - \geo{\lambda_2} \geo{h} \vtext n_1
    \end{pmatrix},
\end{equation}
with $\alpha = -1 + \sqrt{3}$. The auxiliary variables and geometric means are now
\begin{equation}
    \begin{aligned}
        c = \sqrt{gh}, \quad
        \cext = \sqrt{g\hext}, \quad
        \vnext = n_1 v_1^\ext + n_2 v_2^\ext, \quad
        \vtext = -n_2 v_1^\ext + n_1 v_2^\ext,\\
        \geo{h} = \sqrt{h\hext}, \quad
        \geo{\lambda_1} = \sqrt{(|v_n| + c) (|\vnext| + \cext)}, \quad 
        \geo{\lambda_2} = \sqrt{|v_n| |\vnext|}.
    \end{aligned}  
\end{equation}
As in Sec.~\ref{sec:subcrit-outflow}, we select the positive value of $\alpha$.
The numerical boundary flux \eqref{eq:subcrit-inflow} is also consistent, as can be shown by taking $\Uext = \Uvec$, recalling that $\alpha^2 + 2 \alpha - 2 = 0$. 

\subsubsection{Supercritical inflow numerical boundary flux}\label{sec:supercrit-inflow}

Supercritical inflow boundaries are characterized by $v_n < 0$ and $|v_n| < c$.
From the diagonal matrix \eqref{eq:switch-matrix} there are three negative entries so
\begin{equation}
\absLambM = \text{diag}(|v_n - c|,\,|v_n|,\,|v_n + c|),
\end{equation}
and $\IMinus = \text{diag}(1,\,1,\,1)$.
The incoming characteristic variables at a supercritical inflow boundary are
\begin{equation}
    \Wminus = \frac{c}{2\sqrt{g}}\begin{pmatrix}
        \alpha c - v_n\\\sqrt{2}v_\tau\\ \alpha c + v_n
    \end{pmatrix}.
\end{equation}
We weakly impose the data with the boundary vector ansatz \eqref{eq:GvecAnsatz}, external evolution variables \eqref{eq:primVarsAnsatz}, and
\begin{equation}
    \absLambMExt = \text{diag}(|\vnext - \cext|,\, |\vnext|,\, |\vnext + \cext|)
\end{equation}
to find in \ref{app:supercrit-inflow} that the numerical boundary flux function for subcritical outflow that satisfies \eqref{eq:swe_relation2} is
\begin{equation}
    \label{eq:supercrit-inflow}
      \statevec{F}^*_n(\Uvec, \Uext)
      =
      \begin{pmatrix}
      (\alpha - 1) h v_n - \frac{\alpha}{2g} \geo{\lambda_1} \cext (\alpha \cext - \vnext) - \frac{\alpha}{2g} \geo{\lambda_3} \cext (\alpha \cext + \vnext)\\\hdashline
      \left(\frac{\alpha}{2} - 1\right) h v_1 v_n
         + (1 - 2\alpha)\frac{g h^2}{2} n_1\\[0.1cm] 
         - \frac{1}{4g} \geo{\lambda_1} \cext (\alpha \cext - \vnext)(\alpha v_1 - 2 c n_1)
         - \frac{1}{4g} \geo{\lambda_3} \cext (\alpha \cext + \vnext)(\alpha v_1 + 2 c n_1)\\[0.1cm]
         + \geo{\lambda_2} \geo{h} \vtext n_2
\\\hdashline
      \left(\frac{\alpha}{2} - 1\right) h v_2 v_n
         + (1 - 2\alpha)\frac{g h^2}{2} n_2\\[0.1cm] 
         - \frac{1}{4g} \geo{\lambda_1} \cext (\alpha \cext - \vnext)(\alpha v_2 - 2 c n_2)
         - \frac{1}{4g} \geo{\lambda_3} \cext (\alpha \cext + \vnext)(\alpha v_2 + 2 c n_2)\\[0.1cm]
         - \geo{\lambda_2} \geo{h} \vtext n_1
    \end{pmatrix}
\end{equation}
      with $\alpha = -1 + \sqrt{3}$ as well as auxiliary variables and geometric means
\begin{equation}
    \begin{aligned} 
        c = \sqrt{gh}, \quad
        \cext = \sqrt{g\hext}, \quad
        \vnext = n_1 v_1^\ext + n_2 v_2^\ext, \quad
        \vtext = -n_2 v_1^\ext + n_1 v_2^\ext, \quad
        \geo{h} = \sqrt{h\hext},\\
        \geo{\lambda_1} = \sqrt{(|v_n| + c) (|\vnext| + \cext)}, \quad 
        \geo{\lambda_2} = \sqrt{|v_n| |\vnext|}, \quad
        \geo{\lambda_3} = \sqrt{(|v_n| - c) (|\vnext| - \cext)}.
    \end{aligned}
\end{equation}
As in Section~\ref{sec:subcrit-outflow} we select the positive value of $\alpha$ from \eqref{eq:alphas}.
The numerical boundary flux \eqref{eq:supercrit-inflow} is also consistent. 

\section{Numerical results}\label{sec:numerical}

We apply the newly derived boundary fluxes to compute solutions of the Burgers and shallow water equations with DGSEM approximations.
The spatial discretizations for the split form DGSEM are available in Trixi.jl
\cite{ranocha2022adaptive,schlottkelakemper2021purely} and TrixiShallowWater.jl \cite{winters2025trixi}.
For time integration we use CFL-based time stepping with the five-stage, four-order explicit Runge-Kutta method of Carpenter and Kennedy \cite{Carpenter&Kennedy:1994} implemented in
OrdinaryDiffEq.jl \cite{rackauckas2017differentialequations}.
The unstructured curvilinear quadrilateral meshes were constructed
with HOHQMesh \cite{kopriva2024hohqmesh:joss,kopriva2024hohqmeshjl}.
We use Plots.jl \cite{christ2023plots} and ParaView \cite{ahrens2005paraview} to visualize the results.
The implementations of the nonlinearly stable boundary fluxes derived here as well as all source code needed to reproduce the numerical experiments is
available online in our reproducibility repository~\cite{winters2025numericalRepro}.

\subsection{Burgers' equation}\label{sec:num_burgers}

The first, and simplest, test problem approximates the solution of the Burgers equation \eqref{eq:BasicBurgers} on the domain $[-1,1]$.
We manufacture a solution
\begin{equation}
    \label{eq:burgersMMS}
    u(x,t) = 2 + \sin(\pi(x - t) - 0.7),
\end{equation}
so that it remains smooth over time, and introduces a source term
\begin{equation}
    \label{eq:burgersSource}
    s(x,t) = \pi \cos(\pi(x - t) - 0.7) (1 + \sin(\pi(x - t) - 0.7)).
\end{equation}
The solution \eqref{eq:burgersMMS} is always positive in the domain $[-1,1]$.
Thus, a boundary condition is needed at the left.

We take the final time to be $t = 120$, $\text{CFL}=0.75$, and use one of three boundary fluxes to impose the boundary condition: (1) The entropy conservative (EC) flux \cite{gassner2013skew}, (2) The LLF flux \cite{toro2013riemann}, and (3) the new boundary flux from \eqref{eq:num_burgers}.
The computation at all interior element interfaces uses the entropy conservative \revOne{(EC)} flux so the approximation is dissipation-free in the interior, and the only dissipation is introduced at the left and right physical boundaries.

For each test we use a five element mesh with polynomials of degree $N=7$ on each element.
The EC flux is not an upwind Riemann solver and does not consider the characteristics in its design.
This, unsurprisingly, makes it an especially poor choice for open boundary conditions as there is no bound, cf. Rem.~\ref{rem:ECBCFlux}.
We see in Fig.~\ref{fig:burgersEC}(a) that the solution computed with the EC flux at the open boundaries exhibits instability.
Imposing the inflow-outflow boundary conditions with the LLF flux does run successfully to the final time for this problem, see Fig.~\ref{fig:burgersEC}(b) even though the inflow condition is not provably bounded, cf. Rem.~\ref{rem:ECBCFlux}. (There is outflow dissipation, however, to mitigate that fact.)
The boundary flux \eqref{eq:num_burgers} produces Fig.~\ref{fig:burgersEC}(c), which looks identical to the LLF result.
At the final time, the $L_2$ error with the new flux result is $8.80419344\cdot 10^{-7}$ and the LLF result has an $L_2$ error of $8.80425611\cdot 10^{-7}$, which are the same to four significant digits.
\begin{figure}[!ht]
   \centering
    \subfloat[EC flux]
    {
        \includegraphics[width=0.31\textwidth]{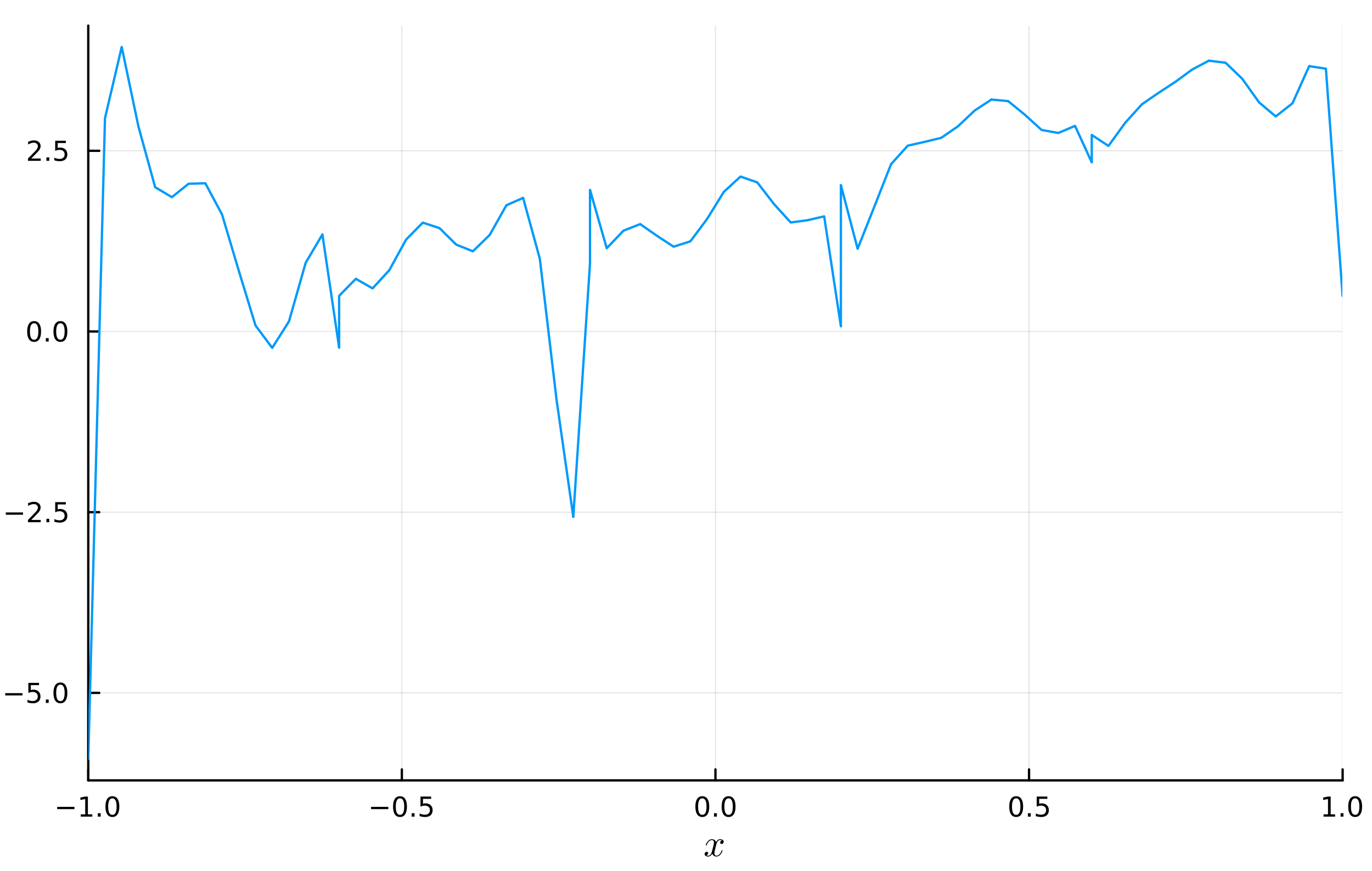}
    }
\subfloat[LLF flux]
    {
        \includegraphics[width=0.31\textwidth]{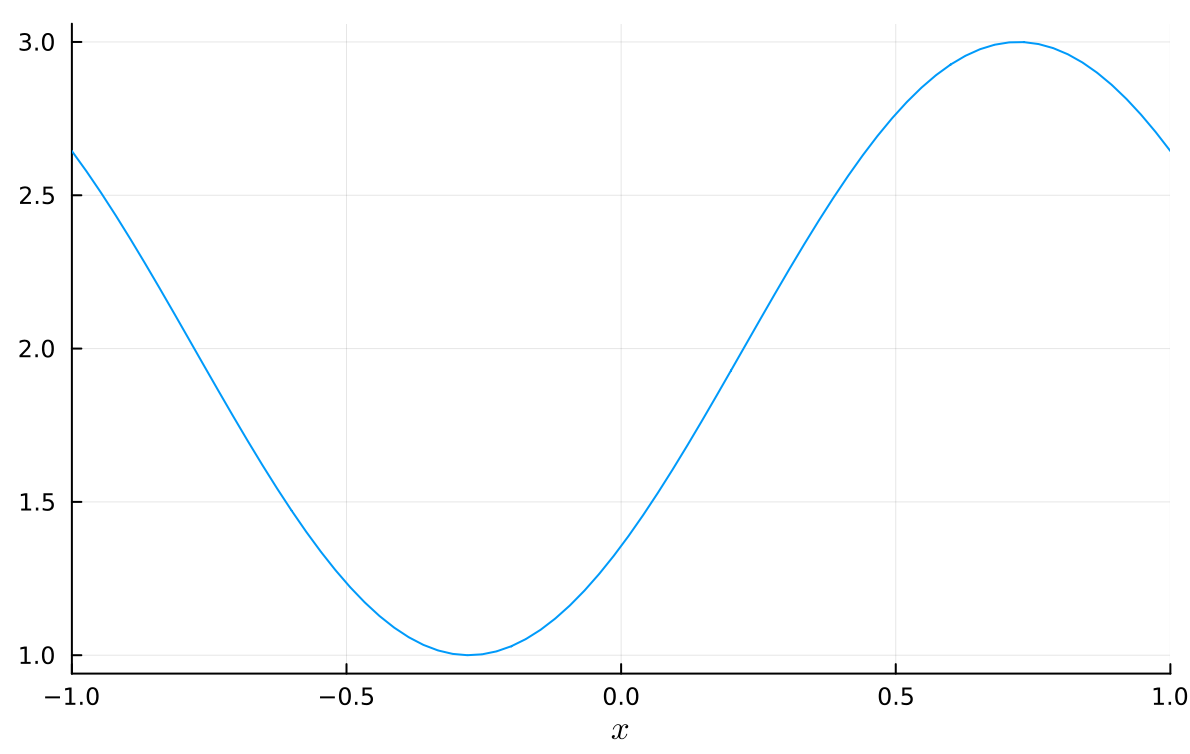}
    }
\subfloat[Nonlinearly stable flux \eqref{eq:num_burgers}]
    {
        \includegraphics[width=0.31\textwidth]{burgers_stable.png}
    }
    \caption{The solution \eqref{eq:burgersMMS} of Burgers' equation at final time $t=120$ with polynomials of degree seven on a five element mesh. \revOne{All element interfaces use the EC flux; so the approximation is dissipation-free in the interior. Thus, the only dissipation in the approximation is introduced at the left and right physical boundaries by the selected boundary fluxes.}}
    \label{fig:burgersEC}
\end{figure}

\subsection{Shallow water equations}

The next two examples approximate the shallow water equations in two space dimensions.
For the first example, we create a manufactured solution to compute subcritical and supercritical flows in a channel.
In the second example, we solve a geostrophic adjustment test problem where the shallow water solution evolves to a rotating equilibrium under Coriolis forces~\cite{castro2008finite,kuo2000nonlinear}.

\subsubsection{Manufactured solution in a channel}\label{sec:swe_channel}

For the first example, we compute the shallow water flow with a constant background velocity in a slanted channel through which a Gaussian pulse in the water height enters, propagates through the channel, and leaves.
The problem has inflow at the bottom portion of the channel and outflow at the top.
The edges of the channel are non-penetrating slip walls, which are neutrally stable, with glancing boundary conditions \cite{nordstrom2022linear}.
Fig.~\ref{fig:subcrit_2d}(a) shows the slanted channel domain together with an unstructured quadrilateral mesh.
We curve the inflow and outflow boundaries of the channel to fully exercise the terms in the new boundary fluxes.

The manufactured solution for the shallow water equations written in primitive variables is
\begin{equation}
    \label{eq:mms_swe}
    \begin{pmatrix}
        h\\[0.1cm]v_1\\[0.1cm]v_2
    \end{pmatrix}
        =
        \begin{pmatrix}
            \frac{1}{2g}\left(h_0 + \text{exp}\left(-8\left(\left(x - \frac{t}{\sqrt{2}} + 2 \sqrt{2}\right)^2 + \left(y - \frac{t}{\sqrt{2}} + \frac{1}{\sqrt{2}})^2\right)\right)\right)\right)\\[0.1cm]
            \frac{1}{\sqrt{2}}\\[0.1cm]
            \frac{1}{\sqrt{2}}
        \end{pmatrix},
\end{equation}
which introduces a corresponding source term $\statevec{s}(x,t) = (0\,,\,g h h_x \,,\, g h h_y)^T$.
The real constant $h_0$ is left free to adjust the manufactured solution to be either in the subcritical or supercritical flow regime.

For the tests, we use polynomials of degree $N=5$ in each spatial direction on each element.
The mesh, shown in Fig.~\ref{fig:subcrit_2d}(a), contains 140 quadrilateral elements.
We take the gravitational constant $g=9.81$ and $\text{CFL}=0.9$.
The final simulation time is $t_{\text{end}}=11$, during which the Gaussian pulse enters, propagates through, and exits the domain.

For subcritical flows, we take $h_0=32$ in \eqref{eq:mms_swe}, which corresponds to a normal Froude number $|v_n|/\sqrt{gh} \approx 0.25$.
In this way, the subcritical inflow-outflow fluxes from \eqref{eq:subcrit-inflow} and \eqref{eq:subcrit-outflow} can be exercised.
We use the EC flux \cite{wintermeyer2017entropy} at element interfaces, so the only dissipation introduced into the approximation is due to the boundary fluxes.
With the new fluxes, the configuration runs the manufactured solution simulation successfully to the final time.
The $L_2$ errors of the conserved variables for the new fluxes used at the inflow-outflow boundaries with the dissipation-free EC flux at interior interfaces are $3.13 \cdot 10^{-5}$ for $h$, $1.12 \cdot 10^{-3}$ for $hv_1$ and $1.10 \cdot 10^{-3}$ for $hv_2$. We show the numerical solution of the subcritical configuration at time $t=6$ in Fig.~\ref{fig:subcrit_2d}(b).

Next, we set $h_0=\frac{3}{5}$ in \eqref{eq:mms_swe} to test a supercritical flow with a normal Froude number $|v_n|/\sqrt{gh} \approx 1.8$,
so that the supercritical inflow-outflow fluxes from \eqref{eq:supercrit-inflow} and \eqref{eq:supercrit-outflow} can be exercised.
Elsewhere, we set the internal element interfaces to use the EC flux \cite{wintermeyer2017entropy} so that the only dissipation is introduced by the open boundary treatments.
The computations successfully run through to the final time of $t_{end} = 11$.
We show the numerical solution for the supercritical configuration at time $t=6$ in Fig.~\ref{fig:subcrit_2d}(c).
As before, we present the $L_2$ errors of the conserved variables with the dissipation-free EC flux at interior interfaces, $9.71 \cdot 10^{-5}$ for $h$, $8.16 \cdot 10^{-5}$ for $hv_1$ and $7.74 \cdot 10^{-5}$ for $hv_2$.
\begin{figure}[!ht]
    \centering
    \subfloat[Domain and mesh]
    {
       \includegraphics[width=0.3175\textwidth]{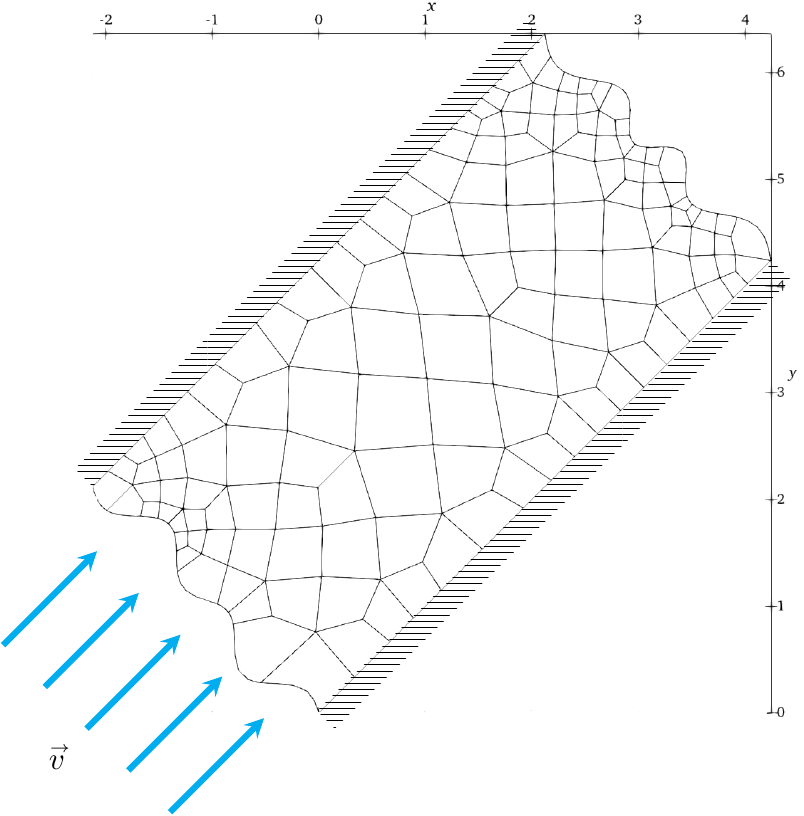}
    }
    \subfloat[Subcritical fluxes \eqref{eq:subcrit-outflow}, \eqref{eq:subcrit-inflow}]
    {
        \includegraphics[width=0.3175\textwidth, trim={60 10 10 35}, clip]{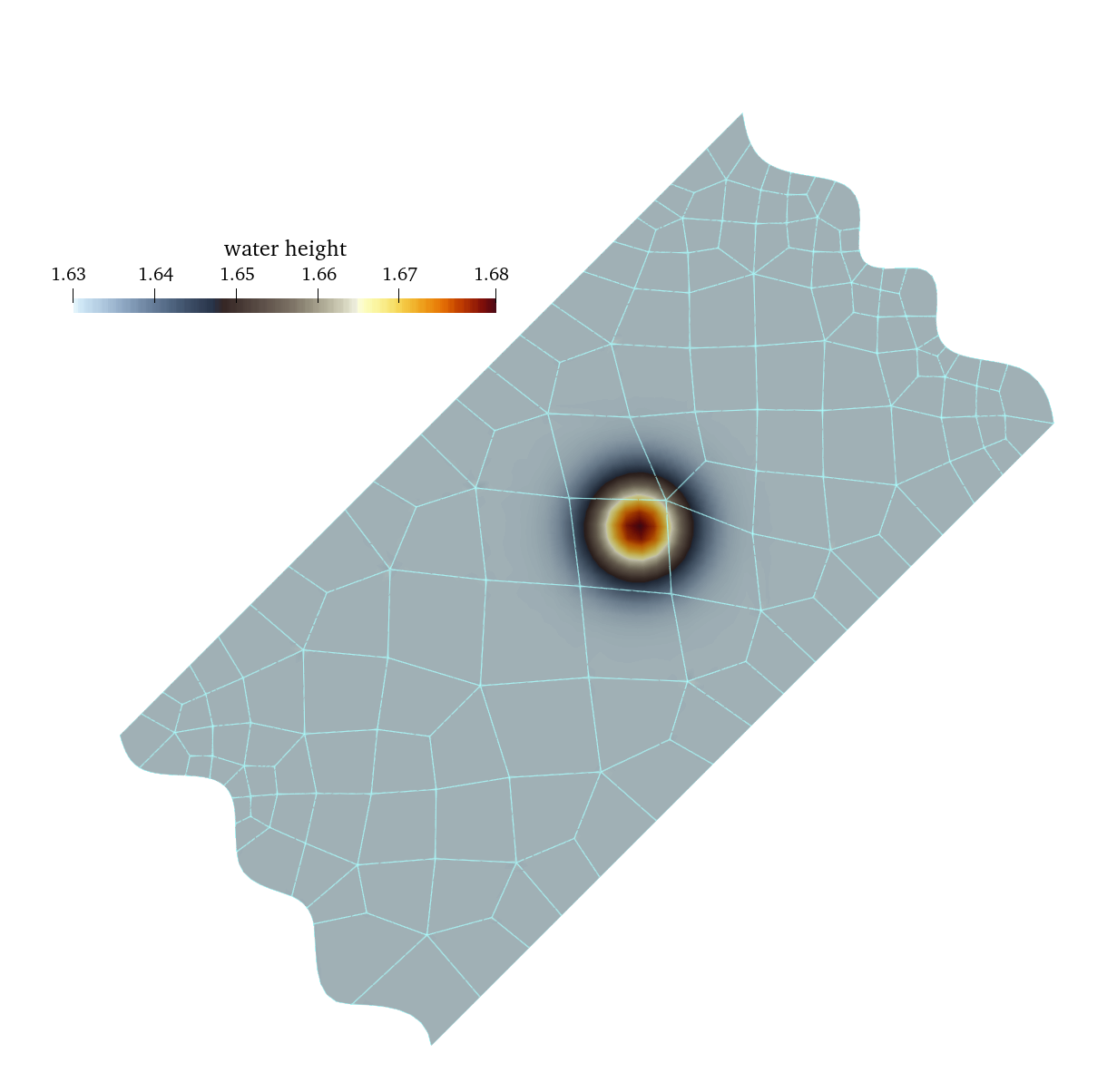}
    }
    \subfloat[Supercritical fluxes \eqref{eq:supercrit-outflow}, \eqref{eq:supercrit-inflow}]
    {
        \includegraphics[width=0.3175\textwidth, trim={50 0 0 35}, clip]{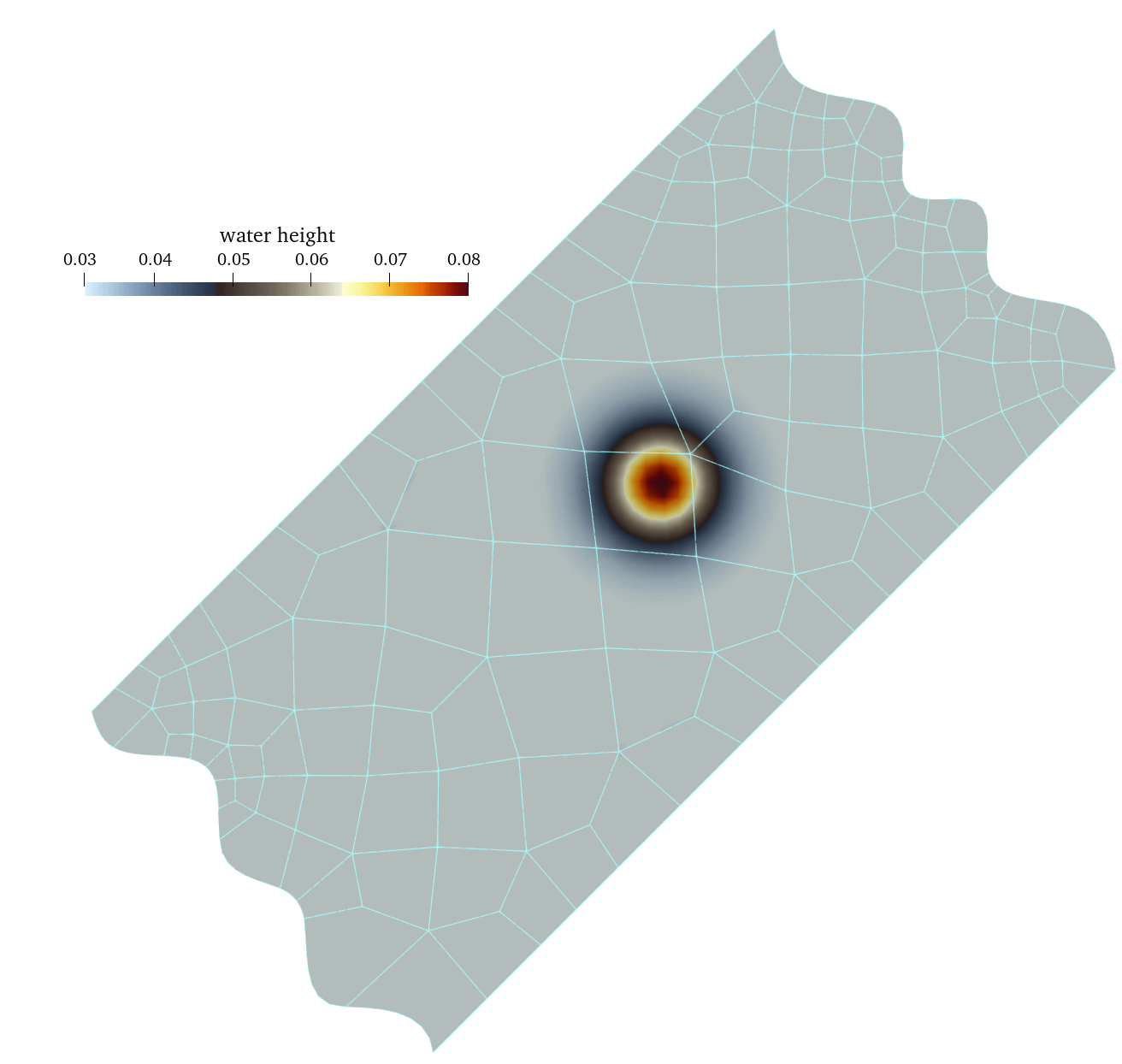}
    }
    \caption{(a): Domain for the solution \eqref{eq:mms_swe} with a constant velocity through a channel, curved inflow-outflow boundaries and slip walls.
    (b): Subcritical solution \eqref{eq:mms_swe} with $h_0 = 32$ at $t=6$. 
    (c): Supercritical solution \eqref{eq:mms_swe} with $h_0 = \frac{3}{5}$ at $t = 6$.
    The mesh contains 140 quadrilateral elements with polynomials of degree $N=5$ in each direction.
    All internal interfaces use the dissipation-free EC flux.}
    \label{fig:subcrit_2d}
\end{figure}

%

\subsubsection{Geostrophic adjustment}\label{sec:swe_geo}

The final problem tests the capability to model geostrophic adjustment, first proposed 
in 
\cite{kuo2000nonlinear}.
The problem is designed for the rotating shallow water equations to examine how initially unbalanced states dynamically evolve to a balanced moving equilibrium.
To pose this problem, we add a Coriolis source term to the shallow water equations \eqref{eq:swe_cons}
\begin{equation}
    \label{eq:corilis}
    \statevec{s}_{\text{Cor}}(\statevec{q})
    =
    \begin{pmatrix}
        0 \\ f h v_2 \\ -f h v_1
    \end{pmatrix},
\end{equation}
where $f = f_0 + \beta y$ is the $\beta$-plane approximation, see \cite{gill1982}.
The Coriolis source term \eqref{eq:corilis} is naturally skew-symmetric \cite{nordstrom2022linear} and does not influence the previously discussed energy/entropy bounds.

The initial conditions are
\begin{equation}
    \label{eq:geo-adjust}
    \begin{pmatrix}
        h\\v_1\\v_2
    \end{pmatrix}
    =
    \begin{pmatrix}
        1 + \frac{A_0}{2}\left(1 - \tanh\left(\frac{\sqrt{(\sqrt{\lambda}x)^2 + (y/\sqrt{\lambda})^2} - R_i}{R_E}\right)\right)\\[0.1cm]
        0\\[0.1cm]
        0
    \end{pmatrix},
\end{equation}
on the domain $\Omega = [-10,10]^2$ and with parameters $A_0=0.5$, $\lambda=2.5$, $R_E = 0.1$, and $R_i = 1$, where the same setup is considered in~\cite{castro2008finite,navas20182d,zhang2024well}.
The gravity and Coriolis parameters are taken to be $g = 1$ and $f = 1$, respectively.
This configuration is subcritical with a normal Froude number of approximately $0.14$ in the region near the equilibrium solution.
Thus, the edges of the square domain are set to be open subcritical outflow boundaries.

The initial water height disturbance from \eqref{eq:geo-adjust} is an unbalanced state that evolves non-axisymmetrically.
The water column falls and strong rotation due to the Coriolis forces generates gravitational waves that propagate outward from the center of the domain.
These gravitational waves should simply leave the domain through the open boundaries, resulting in an elliptical water height shape that slowly rotates clockwise under the Coriolis forces in dynamic equilibrium \cite{kuo2000nonlinear}.
It is the two-dimensional geostrophic adjustment, where initially unbalanced flows transition toward an equilibrium geostrophic balance, that the numerical approximation should reproduce.

For the simulation, we divide the domain $\Omega$ into a $32 \times 32$ Cartesian mesh.
We approximate the solution in each element with polynomials of degree $N = 8$ in each direction.
The final time for the simulation is $t_{end} = 100$ and we take the $\text{CFL}=0.5$.

In Fig.~\ref{fig:geo-novel} we present the solution at six times for the new outflow boundary flux, \eqref{eq:subcrit-outflow} where the dissipation-free EC flux is used at interior interfaces.
So for this case, the interior approximation is nearly dissipation-free and only the boundary flux \eqref{eq:subcrit-outflow} introduces dissipation into the approximation.
\begin{figure}[!ht]
   \centering
    \subfloat[$t=0$]
    {
        \includegraphics[width=0.3\textwidth, trim={115 30 130 30}, clip]{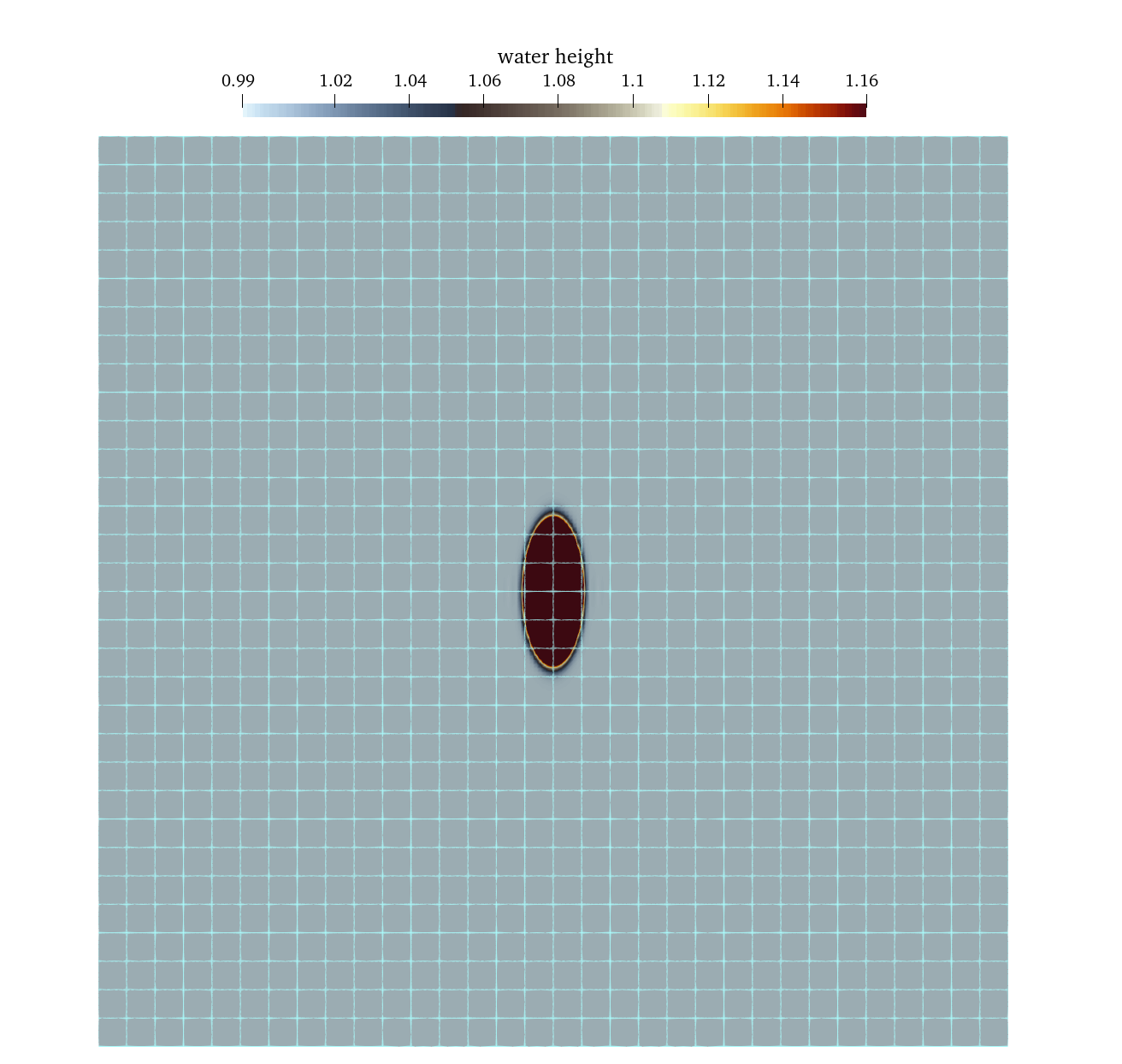}
    }
    \hspace*{0.2cm}
    \subfloat[$t=6$]
    {
        \includegraphics[width=0.3\textwidth, trim={115 30 110 0}, clip]{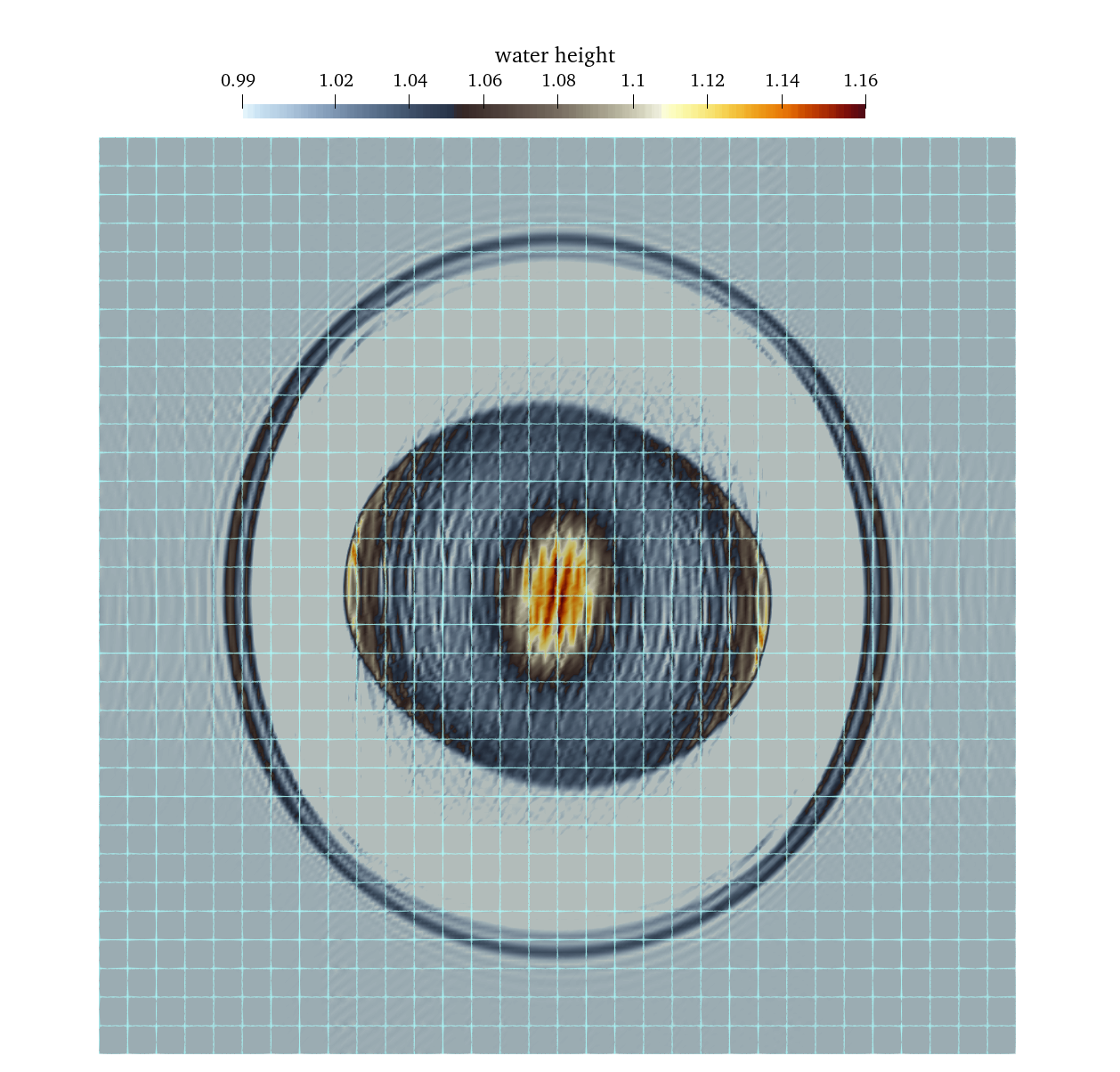}
    }
    \hspace*{0.2cm}
    \subfloat[$t=12$]
    {
        \includegraphics[width=0.3\textwidth, trim={115 30 110 0}, clip]{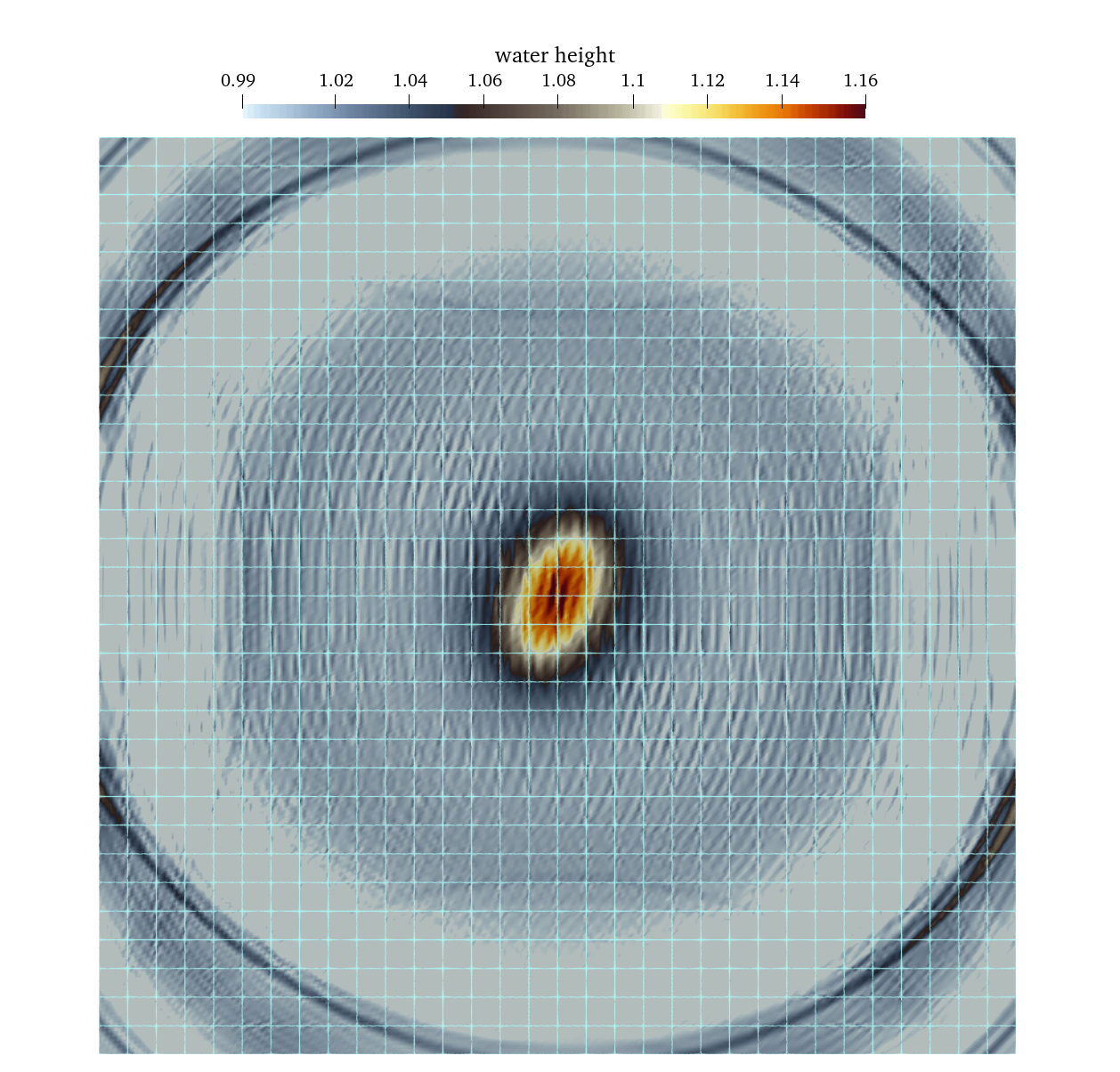}
    }
    \\
    \subfloat[$t=25$]
    {
        \includegraphics[width=0.3\textwidth, trim={115 30 130 30}, clip]{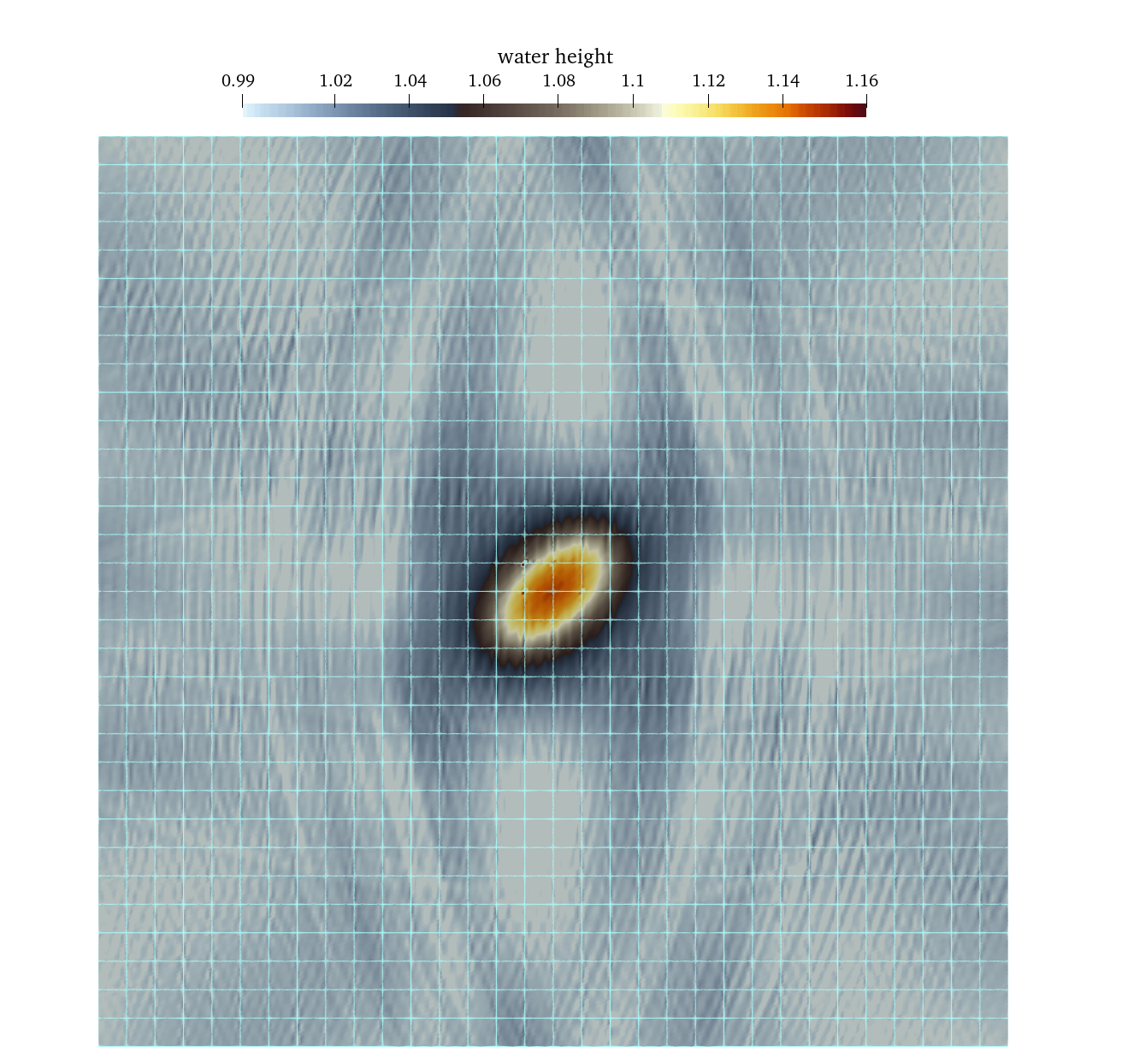}
    }
    \hspace*{0.2cm}
    \subfloat[$t=50$]
    {
        \includegraphics[width=0.3\textwidth, trim={115 30 130 30}, clip]{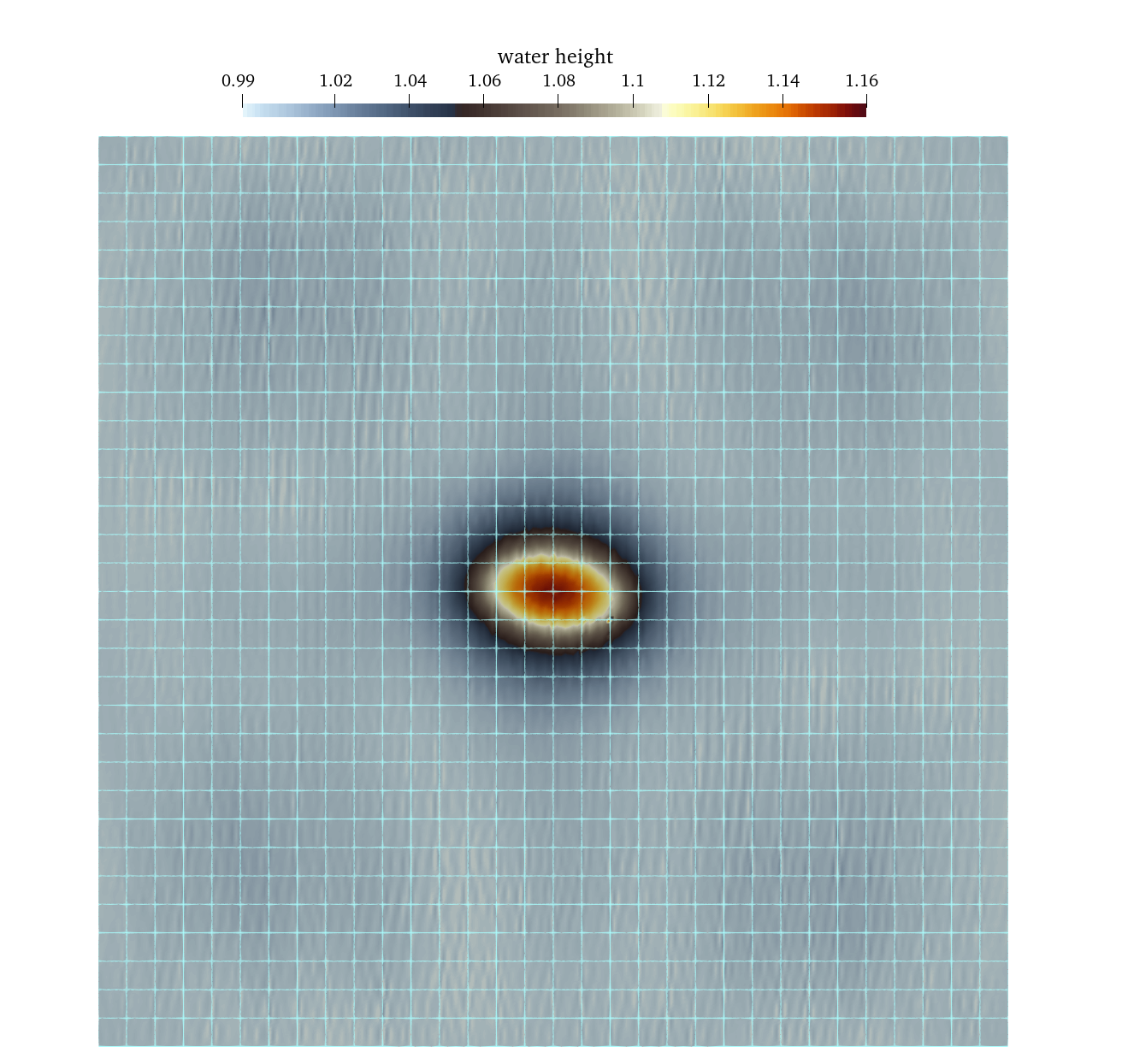}
    }
    \hspace*{0.2cm}
    \subfloat[$t=100$]
    {
        \includegraphics[width=0.3\textwidth, trim={115 30 130 30}, clip]{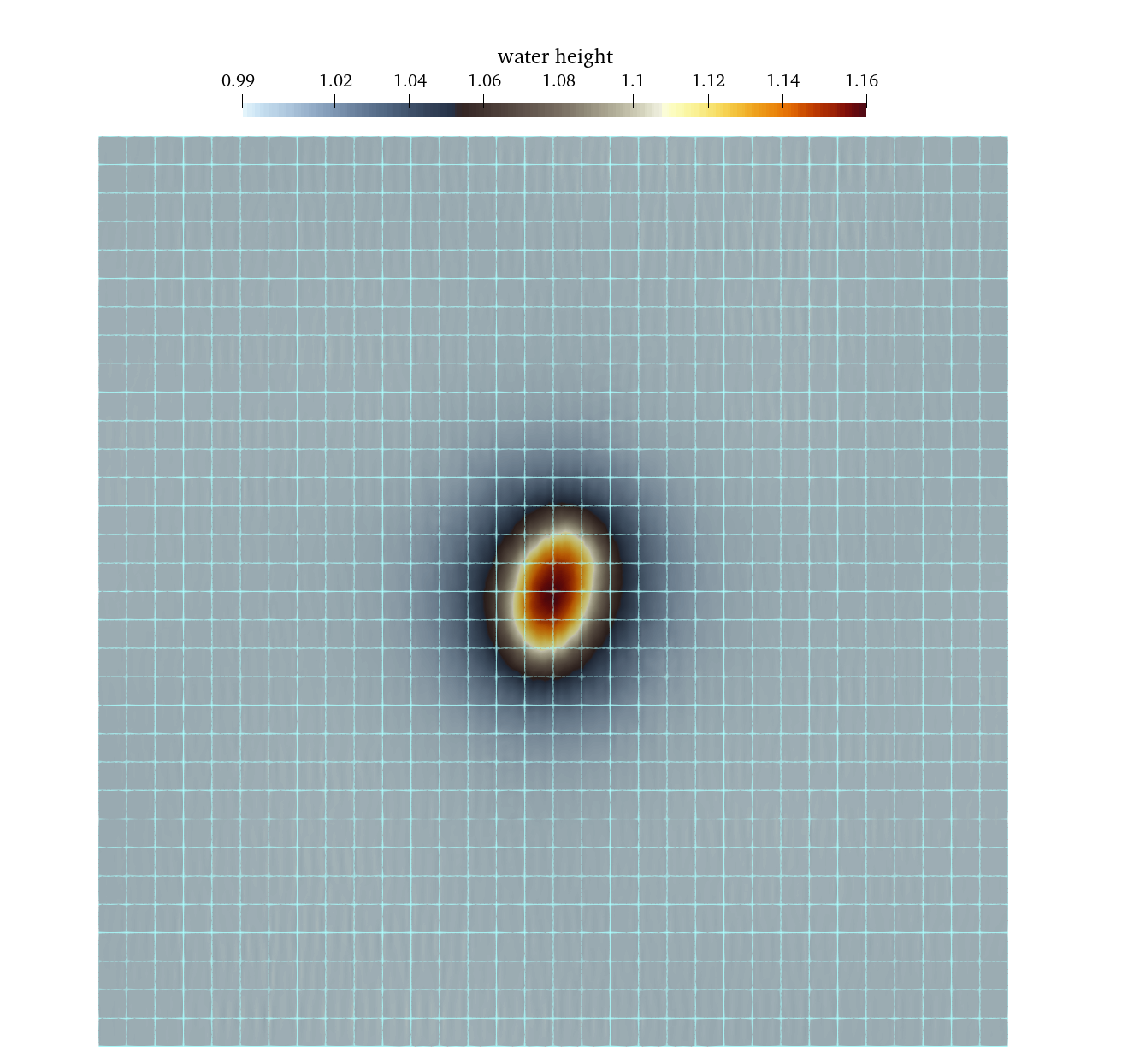}
    }
    \caption{Geostrophic adjustment solution at six times where the subcritical outflow boundary is imposed with the new boundary flux \eqref{eq:subcrit-outflow}
    and the EC flux
    used at interior interfaces.
    The mesh is $32 \times 32$ Cartesian elements with polynomials of degree $N=8$ in each direction.
    (a)--(c): Initial condition and propagation of gravitational waves as they first interact with the outflow boundary where there are some artificial reflections.
    (d)--(f): Rotating, elliptical solution at later times.}
    \label{fig:geo-novel}
\end{figure}
There are some artificial reflections present early on, around $t = 25$, but over time these spurious waves 
propagate out through the outer boundary and the rotating, elliptical solution is attained.
The new boundary flux bounds the solution by data, but no other properties were considered when deriving them.
Despite this, the boundary flux \eqref{eq:subcrit-outflow} strongly damps nonphysical reflections at the outflow boundaries.

We compare the new boundary flux \eqref{eq:subcrit-outflow} to subcritical outflow boundary conditions imposed with tools from linear analysis. 
The theory of characteristics provides a technique to extrapolate internal solution information to establish unknown variables at the boundaries, see \cite{ginting2019central,song2011robust,hou2013robust,yoon2004finite,kuiry2008finite,sleigh1998unstructured} for details.
The (extrapolated) external and internal solution states are then sent into the Riemann solver to impose the linearized outflow boundary condition.

We run the same test problem where the subcritical outflow boundary conditions are imposed with the LLF flux using the linear analysis Riemann invariant approach \cite{song2011robust}.
At interior interfaces, again, we use the dissipation-free EC flux.
For this configuration and a resolution of $32 \times 32$ Cartesian elements with polynomials of degree $N=8$ in each direction, the simulation crashes at $t\approx 19.5$.
This is also true if, instead, we use the more sophisticated HLL \cite{harten1983upstream} flux to impose the outflow boundary conditions.
So, in this situation, neither of the ``classic'' Riemann solvers are able to complete the simulation, whereas the new boundary fluxes ran without issue.
These results highlight (i) the superior performance of the new fluxes and (ii) that one should be cautious when applying linear analysis and approximation techniques to nonlinear problems. 

If we exchange the EC flux for the LLF (or HLL) flux at interior element interfaces to add extra dissipation, the simulations run successfully.
The water height computed with LLF at internal interfaces and weakly imposed outflow boundary via the Riemann invariant boundary conditions \cite{song2011robust} are presented in Fig.~\ref{fig:geo-riemann} at five times.
\begin{figure}[!ht]
   \centering
    \subfloat[$t=0$]
    {
        \includegraphics[width=0.3\textwidth, trim={115 30 130 30}, clip]{novel_t0_geo.png}
    }
    \hspace*{0.2cm}
    \subfloat[$t=6$]
    {
        \includegraphics[width=0.3\textwidth, trim={115 30 110 0}, clip]{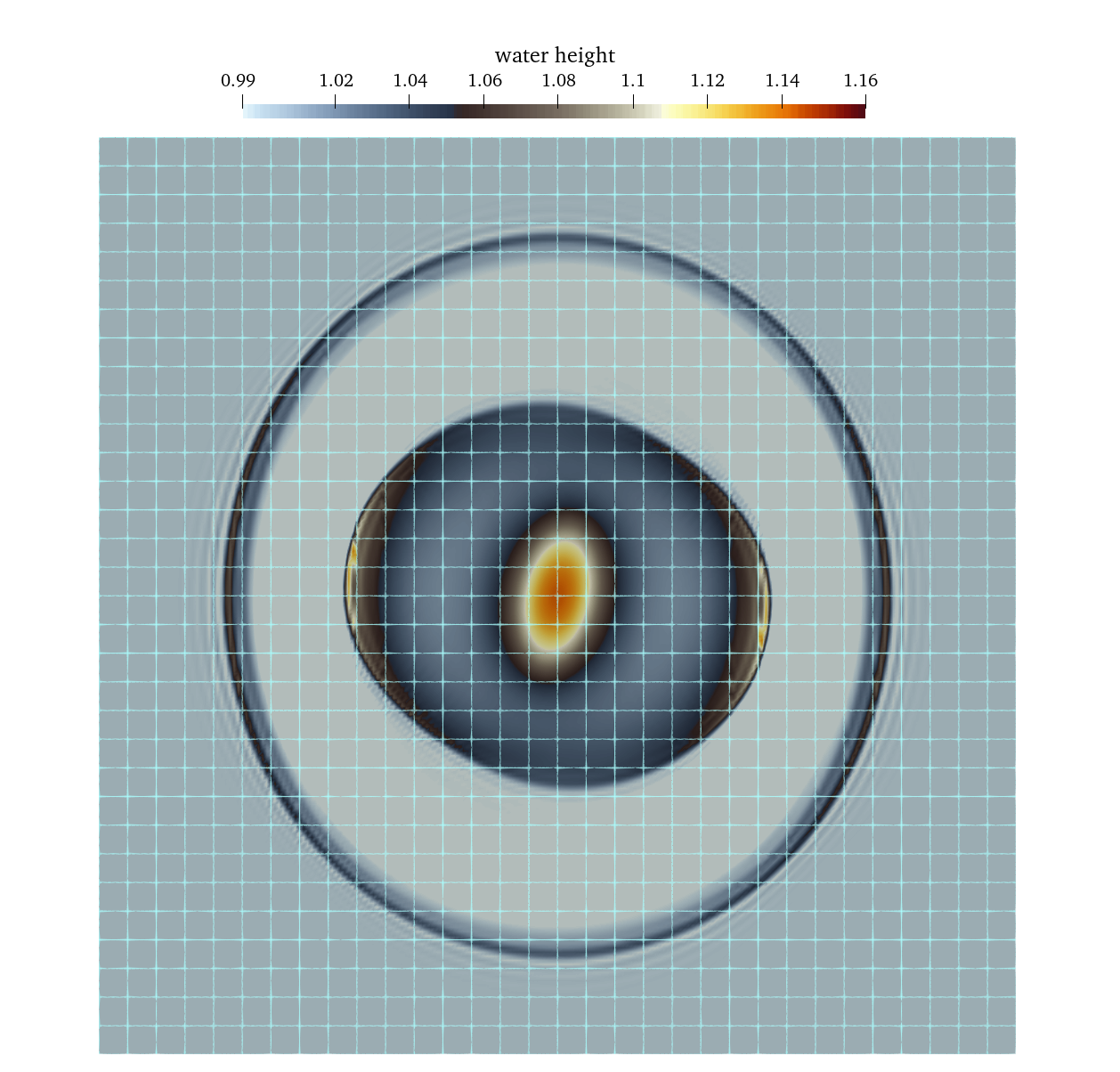}
    }
    \hspace*{0.2cm}
    \subfloat[$t=12$]
    {
        \includegraphics[width=0.3\textwidth, trim={115 30 110 0}, clip]{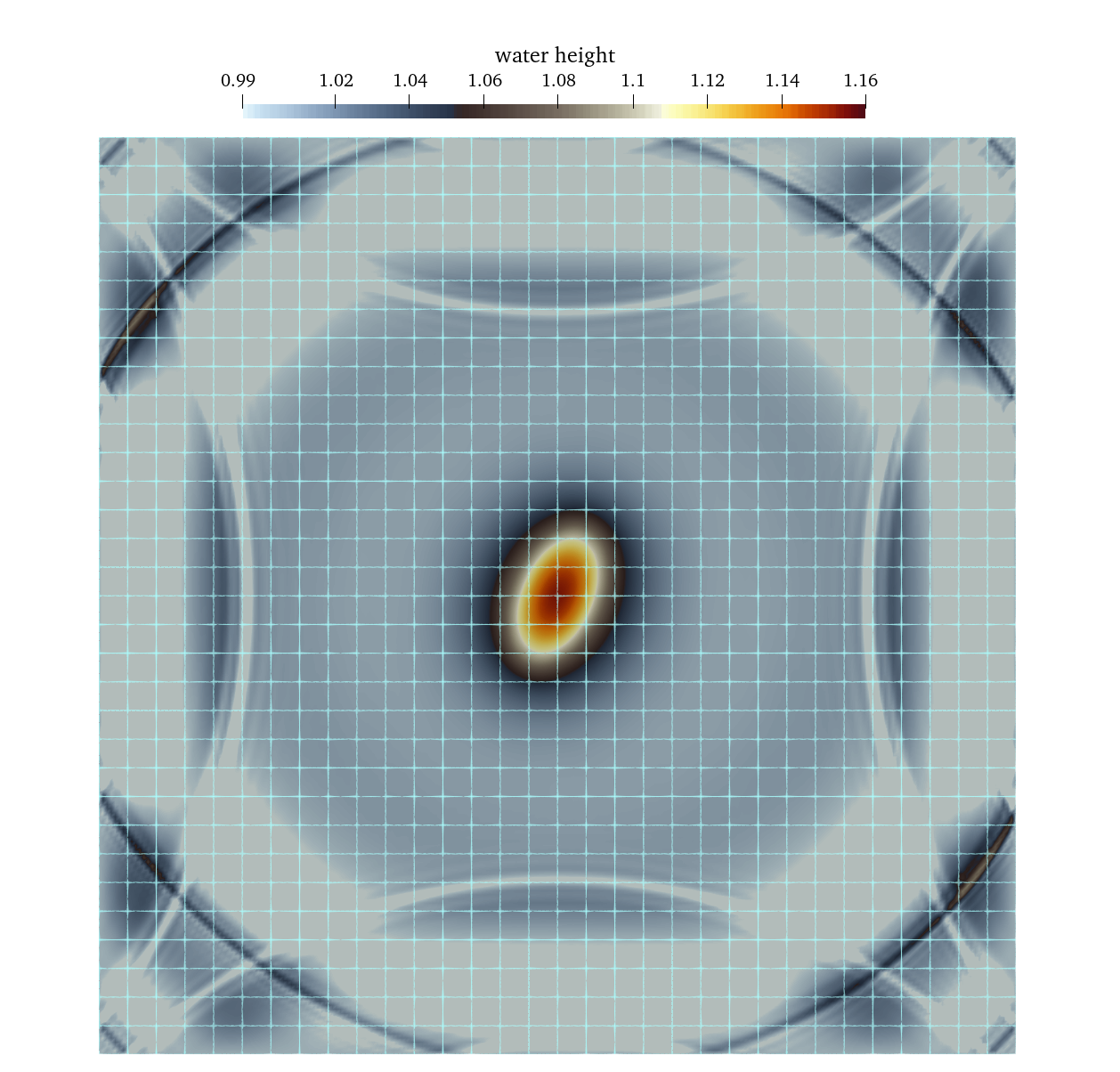}
    }
    \\
    \subfloat[$t=25$]
    {
        \includegraphics[width=0.3\textwidth, trim={115 30 130 30}, clip]{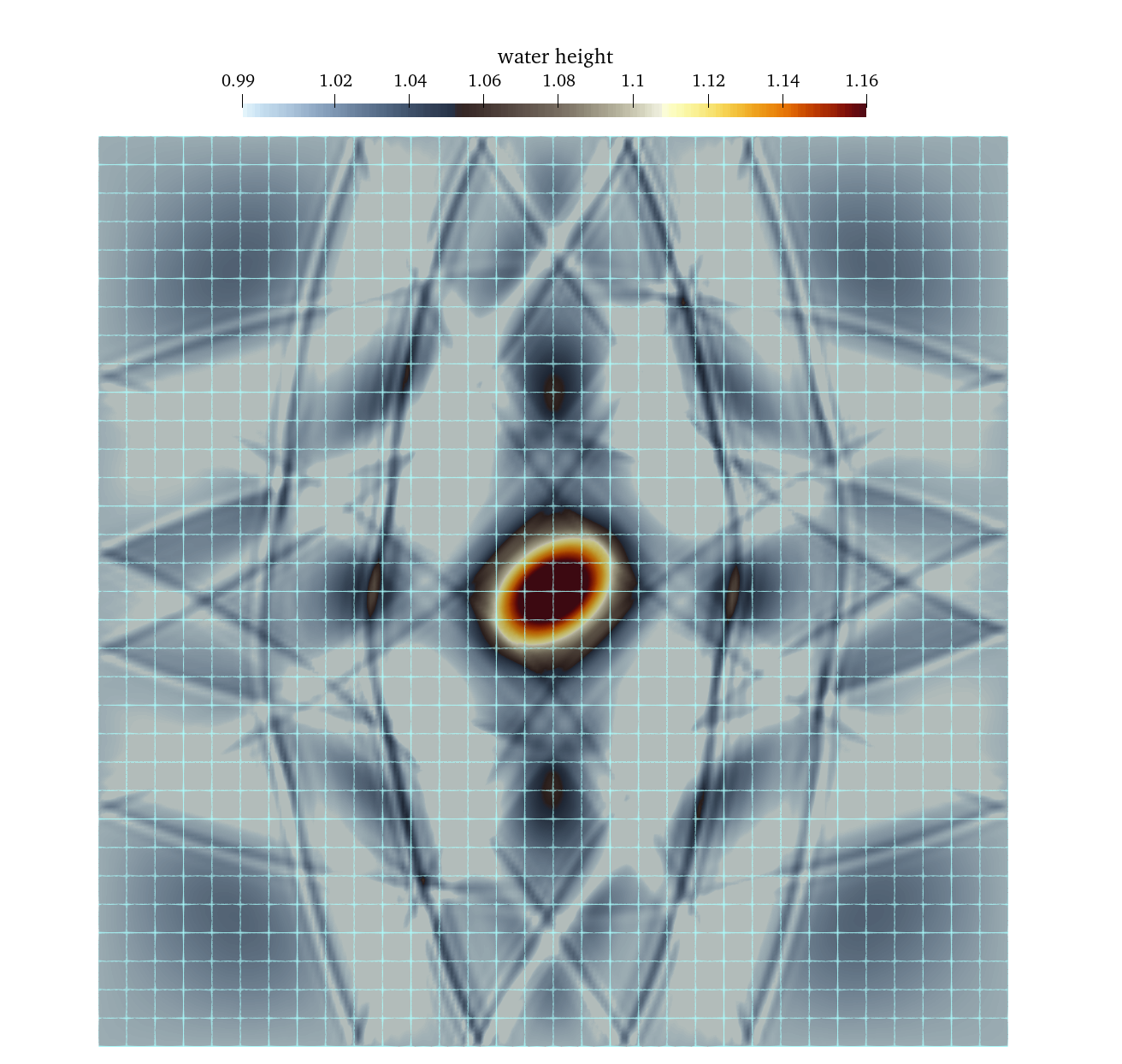}
    }
    \hspace*{0.2cm}
    \subfloat[$t=50$]
    {
        \includegraphics[width=0.3\textwidth, trim={115 30 130 30}, clip]{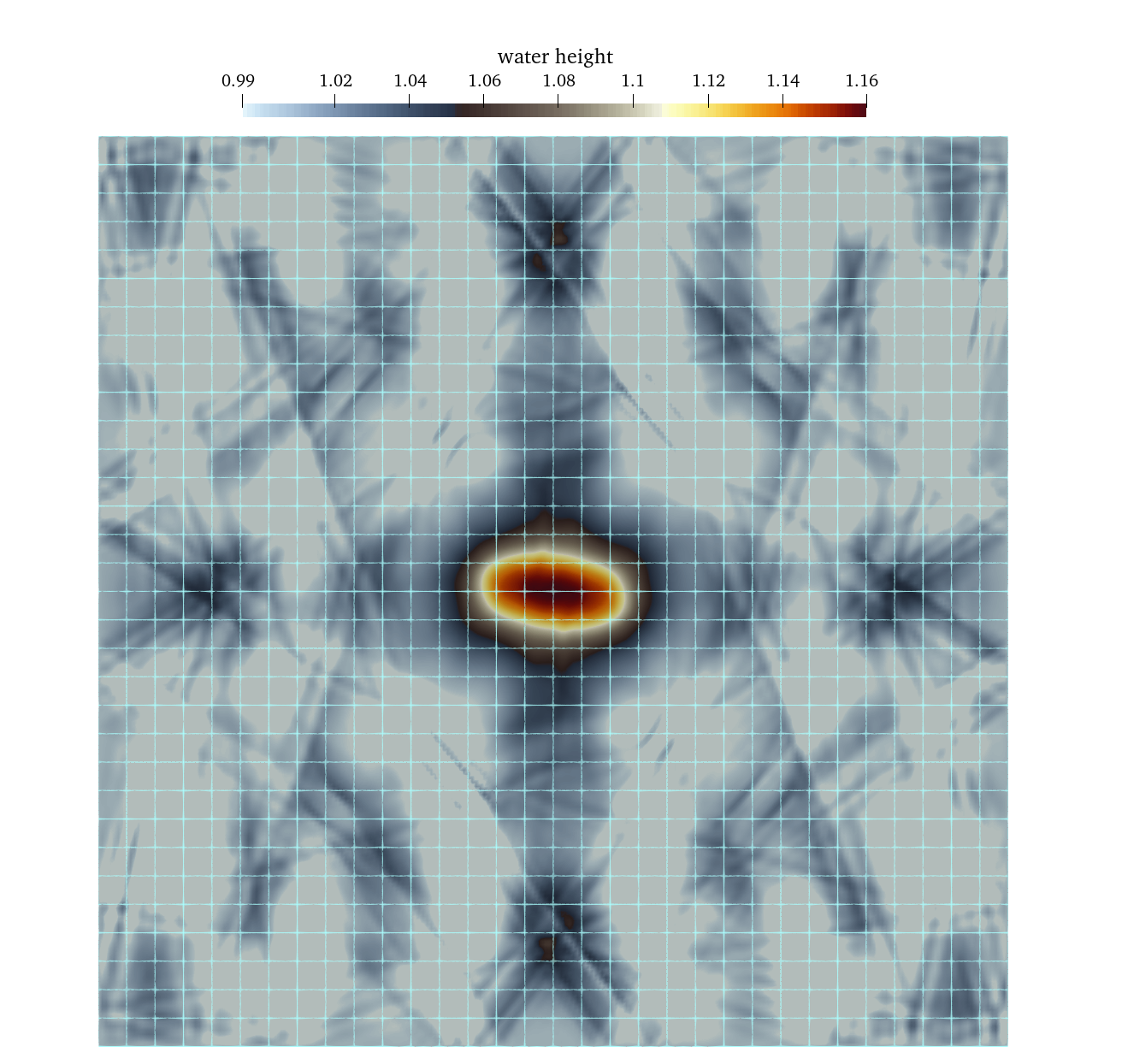}
    }
    \hspace*{0.2cm}
    \subfloat[$t=100$]
    {
        \includegraphics[width=0.3\textwidth, trim={115 30 130 30}, clip]{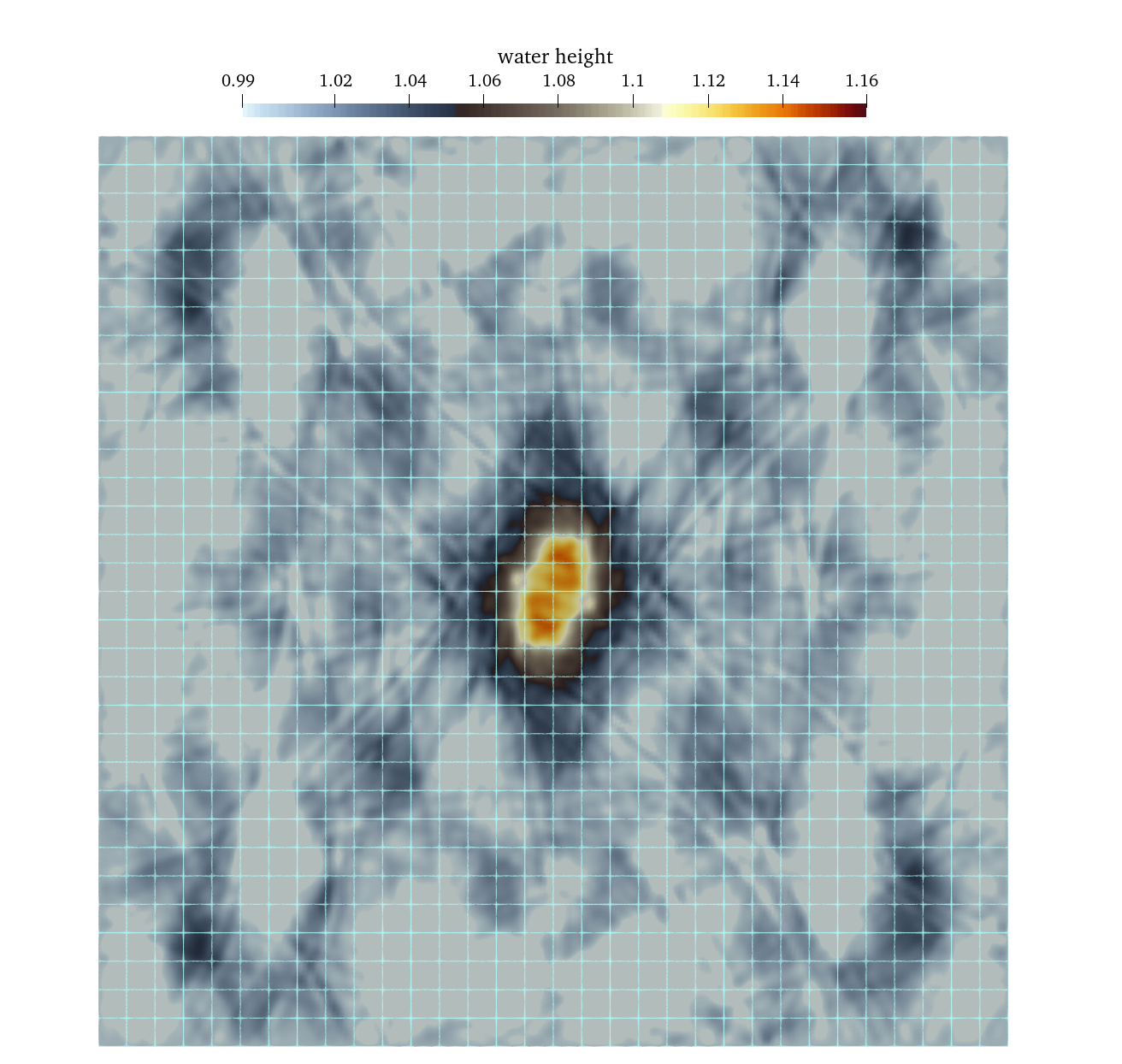}
    }
    \caption{Geostrophic adjustment solution at six times where the subcritical outflow boundary is imposed from linear analysis tools that penalizes the Riemann invariants. The LLF flux is used at the physical boundaries as well as interior interfaces, the latter introduces sufficient dissipation for the simulation to successfully run to the final time. 
    The mesh is $32 \times 32$ Cartesian elements with polynomials of degree $N=8$ in each direction.
    (a)--(c): Initial condition and propagation of gravitational waves as they first interact with the outflow boundary where there are significant artificial reflections.
    (d)--(f): Rotating, elliptical solution at later times polluted with reflections from the outflow boundaries.}
    \label{fig:geo-riemann}
\end{figure}
The flow in Fig.~\ref{fig:geo-riemann} is heavily polluted with spurious waves caused by artificial reflections at the outflow boundaries.
The solution quality obtained from the Riemann invariant boundary conditions and LLF flux in Fig.~\ref{fig:geo-riemann}(d) is far worse compared to the one using the new boundary fluxes (with {\it no} additional internal dissipation) shown in Fig.~\ref{fig:geo-novel}(d).
Similar results were found using the HLL numerical flux at the physical boundaries, and the spurious reflections in the approximate solution remained.
Overall, these results highlight, again, that standard linear tools applied to a smooth nonlinear problem \textit{might} keep a simulation bounded, but there are no guarantees. With the new flux, we have a {\it proven} bound.


\section{Concluding remarks}

We have shown how to interpret nonlinearly stable open boundary treatments as numerical boundary flux functions.
These fluxes penalize (nonlinear) characteristic-type variables at inflow-outflow boundaries.
Incorporated into a high-order split form DGSEM, the nonlinear solution including open boundaries is provably bounded by specified external data.

We explicitly demonstrated how to create such boundary flux functions for the Burgers equation and the two-dimensional shallow water equations.
For the shallow water equations, we identified a congruence transformation that defines the (nonlinear) characteristic variables and preserved the required number of boundary conditions from linear theory as the flow transitions from subcritical to supercritical.

In the numerical tests, we used manufactured solutions for the Burgers and shallow water equations to demonstrate the stability of the new (nonlinear) characteristic boundary fluxes.
We compared these results with weak boundary condition methods from linear analysis and Riemann solvers.
We found that these standard approaches \textit{might} produce bounded solutions, but their success is problem dependent and not guaranteed, unlike the new fluxes, where a \textit{provable} bound exists.

We also tested an equilibrium solution for the rotating shallow water equations, where an elliptical water height disturbance is balanced by Coriolis forces.
At fixed resolution, the solution quality obtained using the new boundary fluxes was significantly improved, with markedly lower artificial reflections, compared to that from boundary states derived from the linear analysis.

Investigating the surprising but welcomed non-reflective properties of the new boundary treatments is the subject of future research.
Additionally, we plan to extend the translation strategy described in this work to other systems of nonlinear hyperbolic conservation laws, such as the compressible Euler equations.

\section*{Acknowledgments}

Andrew Winters was supported by Vetenskapsr{\aa}det, Sweden (award no.~2020-03642 VR).
David Kopriva was supported by grants from the Simons Foundation (\#426393, \#961988).
Jan Nordstr\"{o}m was supported by Vetenskapsr{\aa}det, Sweden [award no. 2021-05484 VR] and University
of Johannesburg Global Excellence and Stature Initiative Funding

\section*{CRediT authorship contribution statement}

{\bf Andrew R. Winters:} Conceptualization; Formal analysis; Methodology; Investigation; Visualization; Software; Writing - original draft
{\bf David A. Kopriva:} Conceptualization; Formal analysis;  Methodology; Writing - original draft
{\bf Jan Nordstr\"{o}m:} Conceptualization; Methodology; Writing - original draft
	
\section*{Data Availability}

A reproducibility repository with necessary instructions and code to reproduce the presented numerical results are available in \cite{winters2025numericalRepro}.
	
\section*{Declaration of competing interest}

The authors declare that they have no known competing financial interests or personal relationships that could have appeared to influence the work reported in this paper.

\appendix



\section[Derivation of the scaling matrix T]{Derivation of the scaling matrix $\TMat$ in \eqref{eq:congruence}}\label{app:congrunet}


We begin with an arbitrary transformation matrix
\begin{equation}
    \TMat
    =
    \begin{pmatrix}
        t_1 & t_2 & t_3\\[0.1cm]
        t_4 & t_5 & t_6\\[0.1cm]
        t_7 & t_8 & t_9\\[0.1cm]
    \end{pmatrix}
\end{equation}
and compute the left side of the relation \eqref{eq:congruence}
\begin{equation}
\label{eq:TMat1}
\resizebox{0.9\linewidth}{!}{$    
    \TMat\,\LambMat\,\TMat^T
    =
   \begin{pmatrix} 
    \left(t_{3}^2+t_{2}^2+t_{1}^2\right)
    v_n+c \left(t_{3}^2-t_{1}^2\right) & \left(t_{3} t_{6}+
    t_{2} t_{5}+t_{1} t_{4}\right)v_n+c \left(t_{3} t_{6}-
    t_{1} t_{4}\right) & \left(t_{3} t_{9}+t_{2} t_{8}+t_{1} t_{7}
    \right)v_n+c \left(t_{3} t_{9}-t_{1} t_{7}\right) \\ 
    \left(t_{3} t_{6}+t_{2} t_{5}+t_{1} t_{4}\right)v_n+c 
    \left(t_{3} t_{6}-t_{1} t_{4}\right) & \left(t_{6}^2+t_{5}^2+t_{4}
    ^2\right)v_n+c \left(t_{6}^2-t_{4}^2\right) & \left(t_{6} 
    t_{9}+t_{5} t_{8}+t_{4} t_{7}\right)v_n+c \left(t_{6} 
    t_{9}-t_{4} t_{7}\right) \\ \left(t_{3} t_{9}+t_{2} t_{8}+t_{1} 
    t_{7}\right)v_n+c \left(t_{3} t_{9}-t_{1} t_{7}\right) & 
    \left(t_{6} t_{9}+t_{5} t_{8}+t_{4} t_{7}\right)v_n+c 
    \left(t_{6} t_{9}-t_{4} t_{7}\right) & \left(t_{9}^2+t_{8}^2+t_{7}
    ^2\right)v_n+c \left(t_{9}^2-t_{7}^2\right)
   \end{pmatrix}.
    $}
\end{equation}

As the value of $c$ is not present on the diagonal of the matrix $\AMat$ in \eqref{eq:congruence} we determine from \eqref{eq:TMat1} that
\begin{equation}
    t_1^2 = t_3^2,\quad t_4^2 = t_6^2,\quad\text{and}\quad t_7^2 = t_9^2.
\end{equation}
This yields several configurations to consider.
If we take $t_3=t_1$ and $t_6=t_4$ or $t_3=-t_1$ and $t_6=-t_4$, after many manipulations, we fail to create a matrix $\TMat$ that satisfies the congruence relation \eqref{eq:congruence}.
So, as a first substitution in \eqref{eq:TMat1} we take
\begin{equation}
    t_3 = t_1, \quad t_6 = -t_4, \quad \text{and} \quad t_9 = t_7,
\end{equation}
to find
\begin{equation}
\label{eq:TMat2}
    \TMat\,\LambMat\,\TMat^T
    =
    \begin{pmatrix} 
    \left(t_{2}^2+2 t_{1}^2\right)v_n & t_{2} t_{5}v_n-2 c t_{1} t_{4} & \left(t_{2} 
    t_{8}+2 t_{1} t_{7}\right)v_n \\ t_{2} t_{5}v_n-2 
    c t_{1} t_{4} & \left(t_{5}^2+2 t_{4}^2\right)v_n & t_{5}
     t_{8}v_n-2 c t_{4} t_{7} \\ \left(t_{2} t_{8}+2 t_{1}
     t_{7}\right)v_n & t_{5} t_{8}v_n-2 c t_{4} t_{7}
     & \left(t_{8}^2+2 t_{7}^2\right)v_n
    \end{pmatrix}.
\end{equation}
Next, we examine the second entry in the first column of \eqref{eq:TMat2} and the target matrix $\AMat$ where we require
\begin{equation}
    t_{2} t_{5}v_n-2 c t_{1} t_{4} = \alpha c.
\end{equation}
This provides two pieces of information
\begin{equation}
    t_2 = 0\quad\text{or}\quad t_5 = 0 \quad\text{and}\quad t_4 = -\frac{\alpha}{2t_1}.
\end{equation}
We take $t_5 = 0$ and substitute the expression for $t_4$ into \eqref{eq:TMat2} to have
\begin{equation}
    \label{eq:TMat3}
    \TMat\,\LambMat\,\TMat^T
    =
    \begin{pmatrix} 
    \left(t_{2}^2+2 t_{1}^2\right) v_n & \alpha c & \left(t_{2} t_{8}+2 t_{1} t_{7}\right)
     v_n \\ \alpha c & \frac{\alpha^2 v_n}{2 t_{1}^2} & \frac{\alpha c 
    t_{7}}{t_{1}} \\ \left(t_{2} t_{8}+2 t_{1} t_{7}\right) v_n
     & \frac{\alpha c t_{7}}{t_{1}} & \left(t_{8}^2+2 t_{7}^2\right) v_n
    \end{pmatrix}.
\end{equation}
Comparing the second diagonal element of \eqref{eq:TMat3} to \eqref{eq:congruence}
gives
\begin{equation}
    \frac{\alpha^2 v_n}{2 t_{1}^2} = v_n \quad \Rightarrow \quad t_1^2 = \frac{\alpha^2}{2},
\end{equation}
and the second entry in the third column yields
\begin{equation}
    \frac{\alpha c t_{7}}{t_{1}} = 0 \quad\Rightarrow\quad t_7=0,
\end{equation}
since we now know that $t_1$ is nonzero.
Substituting these two pieces of information, \eqref{eq:TMat3} simplifies to become
\begin{equation}
    \label{eq:TMat4}
    \TMat\,\LambMat\,\TMat^T
    =
    \begin{pmatrix} 
    \left(t_{2}^2+\alpha^2\right) v_n & \alpha c & t_{2} t_{8} v_n \\ 
    \alpha c & v_n & 0 \\ 
    t_{2} t_{8} v_n & 0 & t_{8}^2 v_n
    \end{pmatrix}.
\end{equation}
From the first entry in the first column of \eqref{eq:TMat4} and \eqref{eq:congruence} we see
\begin{equation}
    \left(t_{2}^2+\alpha^2\right) v_n = \alpha^2 v_n \quad\Rightarrow\quad t_2 = 0,
\end{equation}
and the third entry in the third column gives
\begin{equation}
    t_8^2 v_n = v_n \quad\Rightarrow\quad t_8 = \pm 1.
\end{equation}

Collecting the information for the $\TMat$ matrix entries,
\begin{equation}
    \label{eq:T-choices}
    \begin{aligned}
    t_1 &= t_3 = \pm\frac{\alpha}{\sqrt{2}}, \quad t_4 = -\frac{\alpha}{2t_1} = \mp \frac{1}{\sqrt{2}}, \quad t_6 = -t_4 = \pm\frac{1}{\sqrt{2}},\\[0.1cm]
    t_8 &= \pm 1, \quad t_2 = t_5 = t_7 = t_9 = 0.
    \end{aligned}
\end{equation}
As a final step, we make the choice of positive $t_1$ and $t_8$ values that propagates through the remaining terms in \eqref{eq:T-choices} to arrive at the transformation matrix stated in \eqref{eq:TMat}.

\section{Shallow water numerical boundary flux derivations}

In the derivations below it is convenient to restate and expand the relationship \eqref{eq:swe_relation2}
\begin{equation}
\begin{aligned}
   \label{app:commonCondition}
   \statevec{F}^*_n - \statevec{F}_n &= 2\invMMatT\NMat^T\SMat\,\TMat\,\IMinus\sqrtLambM\left(\sqrtLambM\Wminus - \Gvec\right)\\
   &= 2\invMMatT\NMat\SMat\,\TMat\,\IMinus\absLambM\Wminus - 2\invMMatT\NMat^T\SMat\,\TMat\,\IMinus\sqrtLambM\Gvec
   \\
   &= \text{\textcircled{1}} - \text{\textcircled{2}}.
\end{aligned}   
\end{equation}
First, we consider the term \textcircled{1} to identify the internal flux contributions $\statevec{F}_n$. Then we incorporate the boundary data from term \textcircled{2} that includes the external data vector $\Gvec$.

\subsection[Derivation of subcritical outflow boundary flux]{Derivation of subcritical outflow boundary flux \eqref{eq:subcrit-outflow}}\label{app:subcrit-outflow}

For subcritical outflow regime we have that $v_n > 0$ and $v_n < c$ where $c = \sqrt{gh}$.
From the diagonal matrix \eqref{eq:switch-matrix} there is one negative term so $\IMinus = \text{diag}(1,\,0,\,0)$ and the incoming characteristic variable is
\begin{equation}
    \Wminus = \frac{c}{2\sqrt{g}}\begin{pmatrix}
        \alpha c - v_n\\[0.1cm]
        0\\[0.1cm]
        0
    \end{pmatrix}.
\end{equation}
From the definition of the absolute value we have
\begin{equation}
    |v_n - c| = -(v_n - c) = c - v_n,
\end{equation}
so that
\begin{equation}
\absLambM = \text{diag}(|v_n - c|,\,0,\,0) = \text{diag}(c - v_n,\,0,\,0).
\end{equation}
We expand the first term in \eqref{app:commonCondition}, collect like terms, and apply the forms of $\IMinus$, $\absLambM$, and $\Wminus$ above to obtain
\begin{equation}
    \label{app:subcritTerm1}
    \text{\textcircled{1}}
    =
    \frac{c}{2g}
    \begin{pmatrix}
        \alpha v_n^2 + \alpha^2 c^2 + \alpha c v_n\\\hdashline
        \frac{\alpha}{2} v_1 v_n^2 - \frac{\alpha}{2} c v_1 v_n + \frac{\alpha^2}{2} c^2 v_1 + c v_1 v_n + \alpha c v_1 v_n + \alpha c^2 v_n n_1 + c^2 v_n n_1 - c v_n^2 n_1 + \frac{\alpha^2}{2} c^3 n_1\\\hdashline
        \frac{\alpha}{2} v_2 v_n^2 - \frac{\alpha}{2} c v_2 v_n + \frac{\alpha^2}{2} c^2 v_2 + c v_2 v_n + \alpha c v_2 v_n + \alpha c^2 v_n n_2 + c^2 v_n n_2 - c v_n^2 n_2 + \frac{\alpha^2}{2} c^3 n_2
    \end{pmatrix}
    -
    \statevec{F}_n.
\end{equation}
This identifies the internal flux contributions in the normal direction, $\statevec{F}_n$, as well as remaining terms to be built into the boundary flux function.

Next, we consider the boundary vector ansatz \eqref{eq:GvecAnsatz} and external rotated and scaled primitive variables  \eqref{eq:primVarsAnsatz}.
For the subcritical outflow regime we have
\begin{equation}
    \absLambMExt = \text{diag}(\cext - \vnext,\, 0,\, 0),
\end{equation}
so that
\begin{equation}
    \Gvec = \sqrt{\absLambMExt}\WminusExt
    =
    \frac{\cext}{2\sqrt{g}}\begin{pmatrix}
        \sqrt{\cext - \vnext}(\alpha \cext - \vnext)\\0\\0
    \end{pmatrix}.
\end{equation}
To simplify the presentation we rewrite terms with the geometric mean
\begin{equation}
    \sqrt{\absLambM\,\absLambMExt}
    =
    \text{diag}
    \left(
    \sqrt{(c - v_n)(\cext - \vnext)},\,0,\,0
    \right)
    =
    \text{diag}
    \left(
    \geo{c - v_n},\,0,\,0
    \right).
\end{equation}
In this notation, the second term from \eqref{app:commonCondition} is
\begin{equation}
    \label{app:subcritTerm2}
    \text{\textcircled{2}}
    =
    \geo{c - v_n} \frac{\cext}{2g}
    \begin{pmatrix}
        \alpha(\alpha\cext - \vnext)\\[0.1cm]
        \frac{1}{2}(\alpha\cext-\vnext)(\alpha v_1 - 2cn_1)\\[0.1cm]
        \frac{1}{2}(\alpha\cext-\vnext)(\alpha v_2 - 2cn_2)        
    \end{pmatrix}.
\end{equation}
We insert the contributions from \eqref{app:subcritTerm1} and \eqref{app:subcritTerm2} into \eqref{app:commonCondition} and find
\begin{equation}
   \begin{aligned}
   \statevec{F}^*_n - \statevec{F}_n
   &=
    \text{\textcircled{1}} - \text{\textcircled{2}} \\
    &=
    \frac{c}{2g}
    \begin{pmatrix}
        \alpha v_n^2 + \alpha^2 c^2 + \alpha c v_n\\\hdashline
        \frac{\alpha}{2} v_1 v_n^2 - \frac{\alpha}{2} c v_1 v_n + \frac{\alpha^2}{2} c^2 v_1 + c v_1 v_n + \alpha c v_1 v_n + \alpha c^2 v_n n_1 + c^2 v_n n_1 - c v_n^2 n_1 + \frac{\alpha^2}{2} c^3 n_1\\\hdashline
        \frac{\alpha}{2} v_2 v_n^2 - \frac{\alpha}{2} c v_2 v_n + \frac{\alpha^2}{2} c^2 v_2 + c v_2 v_n + \alpha c v_2 v_n + \alpha c^2 v_n n_2 + c^2 v_n n_2 - c v_n^2 n_2 + \frac{\alpha^2}{2} c^3 n_2
    \end{pmatrix}
    \\
    &\quad -
    \geo{c - v_n} \frac{\cext}{2g}
    \begin{pmatrix}
        \alpha(\alpha\cext - \vnext)\\\hdashline
        \frac{1}{2}(\alpha\cext-\vnext)(\alpha v_1 - 2cn_1)\\\hdashline
        \frac{1}{2}(\alpha\cext-\vnext)(\alpha v_2 - 2cn_2)
    \end{pmatrix}
    -
    \statevec{F}_n.
   \end{aligned}
\end{equation}
The terms above that do not involve $\statevec{F}_n$ define the numerical boundary flux function, $\statevec{F}_n^*$. After many algebraic manipulations that use $\alpha^2 + 2\alpha - 2 = 0$, $v_1 = n_1 v_n - n_2 v_\tau$, and $v_2 = n_2 v_n + n_1 v_\tau$, we arrive at the given expression for the numerical boundary flux function given in \eqref{eq:subcrit-outflow}.

\subsection[Derivation of subcritical inflow boundary flux]{Derivation of subcritical inflow boundary flux \eqref{eq:subcrit-inflow}}\label{app:subcrit-inflow}

For subcritical inflow regime we have that $v_n < 0$ and $|v_n| < c$.
From the diagonal matrix \eqref{eq:switch-matrix} there are two negative terms so that $\IMinus = \text{diag}(1,\,1,\,0)$ and the incoming characteristic variables are
\begin{equation}
    \Wminus = \frac{c}{2\sqrt{g}}\begin{pmatrix}
        \alpha c - v_n\\\sqrt{2}v_\tau\\0
    \end{pmatrix}.
\end{equation}
From the definition of the absolute value we have
\begin{equation}
    v_n = -|v_n| \quad\text{and}\quad |v_n - c| = -(-|v_n| - c) = |v_n| + c
\end{equation}
so that
\begin{equation}
\absLambM = \text{diag}(|v_n - c|,\,|v_n|,\,0) = \text{diag}(|v_n| + c,\,|v_n|,\,0).
\end{equation}
Eschewing algebraic details, we expand the first term in \eqref{app:commonCondition}, collect like terms, and apply the form of $\IMinus$, $\absLambM$, and $\Wminus$ above to get
\begin{equation}
    \label{app:subcritInTerm1}
    \text{\textcircled{1}}
    =
    \frac{c}{2g}
    \begin{pmatrix}
        \alpha^2 c^2 - 2 \alpha c |v_n| - \alpha v_n |v_n| - \alpha c v_n
        \\\hdashline
        \frac{\alpha^2}{2} c v_1 |v_n| + \frac{\alpha^2}{2} c^2 v_1
         - \frac{\alpha}{2}v_1 v_n |v_n| - \frac{\alpha}{2}c v_1 v_n
        + \frac{\alpha^2}{2}c^3 n_1 - c v_n |v_n| n_1 + c^2 v_n n_1 - \alpha c^2 |v_n| n_1
        \\\hdashline
        \frac{\alpha^2}{2} c v_2 |v_n| + \frac{\alpha^2}{2} c^2 v_2
         - \frac{\alpha}{2} v_2 v_n |v_n| - \frac{\alpha}{2} c v_2 v_n
        + \frac{\alpha^2}{2} c^3 n_2 - c v_n |v_n| n_2 + c^2 v_n n_2 - \alpha c^2 |v_n| n_2
    \end{pmatrix}
    -
    \statevec{F}_n.
\end{equation}
We identify the internal flux contributions in the normal direction, $\statevec{F}_n$, as well as remaining contributions to be built into the boundary flux function.

Next, we consider the boundary data ansatz \eqref{eq:GvecAnsatz} and external rotated and scaled primitive variables \eqref{eq:primVarsAnsatz}.
For the subcritical inflow regime we have
\begin{equation}
    \absLambMExt = \text{diag}(|\vnext| + \cext,\, |\vnext|,\, 0),
\end{equation}
so that
\begin{equation}
    \Gvec = \sqrt{\absLambMExt}\WminusExt
    =
    \frac{\cext}{2\sqrt{g}}
    \begin{pmatrix}
        \sqrt{|\vnext| +\cext}(\alpha \cext - \vnext)\\[0.1cm]
        \sqrt{2}\sqrt{|\vnext|}\vtext\\[0.1cm]
        0
    \end{pmatrix}.
\end{equation}
To simplify the presentation we, again, rewrite terms with the geometric mean
\begin{equation}
    \begin{aligned}
    \sqrt{\absLambM\,\absLambMExt}
    &=
    \text{diag}
    \left(
    \sqrt{(|v_n| + c)(|\vnext| + \cext)},\,\sqrt{|v_n|\,|\vnext|},\,0
    \right)\\[0.1cm]
    &=
    \text{diag}
    \left(
    \geo{|v_n| + c},\,\geo{|v_n|},\,0
    \right).
    \end{aligned}
\end{equation}
This notation, together with the simplifying principles
\begin{equation}
    c\,\cext = g\sqrt{h\hext} = g\geo{h}
    ,\quad
    -n_2\vtext = v_1^{\ext} - n_1 \vnext
    ,\quad
    n_1 \vtext = v_2^{\ext} - n_2 \vnext,
\end{equation}
we write the second term from \eqref{app:commonCondition} for subcritical inflow as
\begin{equation}
    \label{app:subcritInTerm2}
    \text{\textcircled{2}}
    =
    \begin{pmatrix}
        \frac{\alpha}{2g}\geo{|v_n| + c}\cext(\alpha\cext - \vnext)\\\hdashline
        \geo{|v_n|}\geo{h} v_1^{\ext} + \frac{\alpha}{4g}\left(\geo{|v_n| + c}\cext v_1 (\alpha \cext - \vnext)\right)
         \\[0.1cm]
        + \frac{1}{2}\geo{|v_n| + c}\geo{h}(\vnext-\alpha\cext) n_1 - \geo{|v_n|}\geo{h} \vnext n_1\\\hdashline
        \geo{|v_n|}\geo{h} v_2^{\ext} + \frac{\alpha}{4g}\left(\geo{|v_n| + c}\cext v_2 (\alpha \cext - \vnext)\right)
         \\[0.1cm]
        + \frac{1}{2}\geo{|v_n| + c}\geo{h}(\vnext-\alpha\cext) n_2 - \geo{|v_n|}\geo{h} \vnext n_2
    \end{pmatrix}
\end{equation}
We combine the contributions from \eqref{app:subcritInTerm1} and \eqref{app:subcritInTerm2} in \eqref{app:commonCondition} to obtain
\begin{equation}
   \begin{aligned}
   \statevec{F}_n^* - \statevec{F}_n
   &=
    \text{\textcircled{1}} - \text{\textcircled{2}} \\
    &=
    \frac{c}{2g}
    \begin{pmatrix}
        \alpha^2 c^2 - 2 \alpha c |v_n| - \alpha v_n |v_n| - \alpha c v_n
        \\\hdashline
        \frac{\alpha^2}{2} c v_1 |v_n| + \frac{\alpha^2}{2} c^2 v_1
         - \frac{\alpha}{2}v_1 v_n |v_n| - \frac{\alpha}{2}c v_1 v_n
        + \frac{\alpha^2}{2}c^3 n_1 - c v_n |v_n| n_1 + c^2 v_n n_1 - \alpha c^2 |v_n| n_1
        \\\hdashline
        \frac{\alpha^2}{2} c v_2 |v_n| + \frac{\alpha^2}{2} c^2 v_2
         - \frac{\alpha}{2} v_2 v_n |v_n| - \frac{\alpha}{2} c v_2 v_n
        + \frac{\alpha^2}{2} c^3 n_2 - c v_n |v_n| n_2 + c^2 v_n n_2 - \alpha c^2 |v_n| n_2
    \end{pmatrix}\\[0.1cm]
    &\quad -
    \begin{pmatrix}
        \frac{\alpha}{2g}\geo{|v_n| + c}\cext(\alpha\cext - \vnext)\\\hdashline
        \geo{|v_n|}\geo{h} v_1^{\ext} + \frac{\alpha}{4g}\left(\geo{|v_n| + c}\cext v_1 (\alpha \cext - \vnext)\right)
         \\[0.1cm]
        + \frac{1}{2}\geo{|v_n| + c}\geo{h}(\vnext-\alpha\cext) n_1 - \geo{|v_n|}\geo{h} \vnext n_1\\\hdashline
        \geo{|v_n|}\geo{h} v_2^{\ext} + \frac{\alpha}{4g}\left(\geo{|v_n| + c}\cext v_2 (\alpha \cext - \vnext)\right)
         \\[0.1cm]
        + \frac{1}{2}\geo{|v_n| + c}\geo{h}(\vnext-\alpha\cext) n_2 - \geo{|v_n|}\geo{h} \vnext n_2
    \end{pmatrix}
    -
    \statevec{F}_n    
    \end{aligned}
\end{equation}
The terms above that do not involve $\statevec{F}_n$ are those that define the numerical boundary flux function, $\statevec{F}_n^*$. After many algebraic manipulations that use $\alpha^2 + 2\alpha - 2 = 0$, $v_1 = n_1 v_n - n_2 v_\tau$, and $v_2 = n_2 v_n + n_1 v_\tau$ we arrive at the given expression for the numerical boundary flux function given in \eqref{eq:subcrit-inflow}. 

\subsection[Derivation of supercritical inflow boundary flux]{Derivation of supercritical inflow boundary flux \eqref{eq:supercrit-inflow}}\label{app:supercrit-inflow}

For supercritical inflow we have that $v_n < 0$ and $|v_n| > c$ where $c = \sqrt{gh}$.
From the diagonal matrix \eqref{eq:switch-matrix} there are three negative terms; so, $\IMinus = \text{diag}(1,\,1,\,1)$ and the incoming characteristic variables are
\begin{equation}
    \Wminus = \frac{c}{2\sqrt{g}}\begin{pmatrix}
        \alpha c - v_n\\[0.1cm]
        \sqrt{2}v_\tau\\[0.1cm]
        \alpha c + v_n
    \end{pmatrix}.
\end{equation}
From the definition of the absolute value we have
\begin{equation}
    v_n = - |v_n|
    ,\quad
    v_n - c = -(-|v_n| - c) = -(|v_n| + c),\quad\text{and}\quad v_n + c = -|v_n| + c = -(|v_n| - c),
\end{equation}
so that
\begin{equation}
\absLambM = \text{diag}(|v_n - c|,\,|v_n|,\,|v_n+c|) = \text{diag}(|v_n| + c,\,|v_n|,\,|v_n| - c).
\end{equation}
We expand the first term in \eqref{app:commonCondition}, collect like terms, and apply the form of $\IMinus$, $\absLambM$, and $\Wminus$ above to obtain
\begin{equation}
    \label{app:supercritTerm1}
    \text{\textcircled{1}}
    =
    \frac{c}{2g}
    \begin{pmatrix}
        2 \alpha^2 c |v_n| + \alpha^2 c v_n\\\hdashline
        \alpha^2 c v_1 |v_n| - \alpha c v_1 |v_n| + c^3(1-2\alpha)n_1\\\hdashline
        \alpha^2 c v_2 |v_n| - \alpha c v_2 |v_n| + c^3(1-2\alpha)n_2
    \end{pmatrix}
    -
    \statevec{F}_n.
\end{equation}
This provides the internal flux contributions in the normal direction, $\statevec{F}_n$, as well as the remaining contributions to be built into the boundary flux function.

Next, we consider the boundary data vector ansatz \eqref{eq:GvecAnsatz} and external rotated and scaled primitive variables \eqref{eq:primVarsAnsatz}.
For the supercritical inflow regime we have
\begin{equation}
    \absLambMExt = \text{diag}(|\vnext| + \cext,\, |\vnext|,\, |\vnext| - \cext),
\end{equation}
so that
\begin{equation}
    \Gvec = \sqrt{\absLambMExt}\WminusExt
    =
    \frac{\cext}{2\sqrt{g}}
    \begin{pmatrix}
        \sqrt{|\vnext| +\cext}(\alpha \cext - \vnext)\\[0.1cm]
        \sqrt{2}\sqrt{|\vnext|}\vtext\\[0.1cm]
        \sqrt{|\vnext| -\cext}(\alpha \cext + \vnext)\\[0.1cm]
    \end{pmatrix}.
\end{equation}
To simplify the presentation we use the geometric mean
\begin{equation}
\begin{aligned}
    \sqrt{\absLambM\,\absLambMExt}
    &=
    \text{diag}
    \left(
    \sqrt{(|v_n| + c)(|\vnext| + \cext)},\,\sqrt{|v_n|\,|\vnext|},\,\sqrt{(|v_n| - c)(|\vnext| - \cext)}
    \right)\\
    &=
    \text{diag}
    \left(
    \geo{|v_n| + c},\,\geo{|v_n|},\,\geo{|v_n|-c}
    \right).
\end{aligned}
\end{equation}
With this notation, the second term from \eqref{app:commonCondition} becomes
\begin{equation}
    \label{app:supercritTerm2}
    \text{\textcircled{2}}
    =
    \begin{pmatrix}
        \frac{\alpha}{2g} \cext\vnext(\geo{|v_n| - c} - \geo{|v_n| + c})
        + \frac{\alpha^2}{2g} \cext^2(\geo{|v_n| - c} + \geo{|v_n| + c})\\\hdashline
        \frac{1}{4g} \alpha\cext\vnext v_1 \geo{|v_n| - c} - \geo{|v_n| + c})\\
        +\frac{1}{4g} \alpha^2\cext^2 v_1 (\geo{|v_n| - c} + \geo{|v_n| + c})
        \\
        + \vnext\geo{|v_n|}\geo{h} - \vnext\geo{|v_n|}\geo{h}  n_1\\
        + \frac{1}{2} \vnext(\geo{|v_n| - c} + \geo{|v_n| + c})\geo{h} n_1\\
        + \frac{1}{2} \alpha \cext(\geo{|v_n| - c} - \geo{|v_n| + c}) \geo{h} n_1\\\hdashline
        \frac{1}{4g} \alpha\cext\vnext v_2 (\geo{|v_n| -c} - \geo{|v_n| + c})\\
        +\frac{1}{4g} \alpha^2\cext^2 v_2 (\geo{|v_n| -c} + \geo{|v_n| + c})
        \\
        + \vnext\geo{|v_n|}\geo{h} - \vnext \geo{|v_n|}\geo{h} n_2\\
        + \frac{1}{2} \vnext(\geo{|v_n| -c} +\geo{|v_n| + c})\geo{h}  n_2\\
        + \frac{1}{2} \alpha \cext(\geo{|v_n| -c} -\geo{|v_n| + c}) \geo{h} n_2
    \end{pmatrix}
\end{equation}
We insert the contributions from \eqref{app:supercritTerm1} and \eqref{app:supercritTerm2} into \eqref{app:commonCondition} to have
\begin{equation}
   \begin{aligned}
   \statevec{F}_n^* - \statevec{F}_n
   &=
    \text{\textcircled{1}} - \text{\textcircled{2}} \\
    &=
        \frac{c}{2g}
    \begin{pmatrix}
        2 \alpha^2 c |v_n| + \alpha^2 c v_n\\\hdashline
        \alpha^2 c v_1 |v_n| - \alpha c v_1 |v_n| + c^3(1-2\alpha)n_1\\\hdashline
        \alpha^2 c v_2 |v_n| - \alpha c v_2 |v_n| + c^3(1-2\alpha)n_2
    \end{pmatrix}
    \\
    &\quad -
    \begin{pmatrix}
        \frac{\alpha}{2g} \cext\vnext(\geo{|v_n| - c} - \geo{|v_n| + c})
        + \frac{\alpha^2}{2g} \cext^2(\geo{|v_n| - c} + \geo{|v_n| + c})\\\hdashline
        \frac{1}{4g} \alpha\cext\vnext v_1 \geo{|v_n| - c} - \geo{|v_n| + c})\\
        +\frac{1}{4g} \alpha^2\cext^2 v_1 (\geo{|v_n| - c} + \geo{|v_n| + c})
        \\
        + \vnext\geo{|v_n|}\geo{h} - \vnext\geo{|v_n|}\geo{h}  n_1\\
        + \frac{1}{2} \vnext(\geo{|v_n| - c} + \geo{|v_n| + c})\geo{h} n_1\\
        + \frac{1}{2} \alpha \cext(\geo{|v_n| - c} - \geo{|v_n| + c}) \geo{h} n_1\\\hdashline
        \frac{1}{4g} \alpha\cext\vnext v_2 (\geo{|v_n| -c} - \geo{|v_n| + c})\\
        +\frac{1}{4g} \alpha^2\cext^2 v_2 (\geo{|v_n| -c} + \geo{|v_n| + c})
        \\
        + \vnext\geo{|v_n|}\geo{h} - \vnext \geo{|v_n|}\geo{h} n_2\\
        + \frac{1}{2} \vnext(\geo{|v_n| -c} +\geo{|v_n| + c})\geo{h}  n_2\\
        + \frac{1}{2} \alpha \cext(\geo{|v_n| -c} -\geo{|v_n| + c}) \geo{h} n_2
    \end{pmatrix}    
    -
    \statevec{F}_n.
    \end{aligned}
\end{equation}
The terms above that do not involve $\statevec{F}_n$ are those that define the numerical boundary flux function, $\statevec{F}_n^*$. After many algebraic manipulations that use $\alpha^2 + 2\alpha - 2 = 0$, $v_1 = n_1 v_n - n_2 v_\tau$, and $v_2 = n_2 v_n + n_1 v_\tau$ we arrive at the given expression for the numerical boundary flux function given in \eqref{eq:supercrit-inflow}.







\bibliographystyle{elsarticle-num}

\end{document}